\documentclass[english, 12pt]{article}
\usepackage{amsmath,amsfonts,amssymb,amsthm}
\usepackage{xcolor}
\usepackage{enumitem}
\usepackage[colorlinks=true,linkcolor=black,citecolor=black,urlcolor=black,bookmarks,breaklinks=true]{hyperref}
\setitemize{itemsep=0.1pt}
\setenumerate{itemsep=0.1pt}

\usepackage{mathtools}

\newtheorem{proposition}{Proposition}[section]
\newtheorem{theorem}[proposition]{Theorem}
\newtheorem{corollary}[proposition]{Corollary}
\newtheorem{lemma}[proposition]{Lemma}

\theoremstyle{definition}

\theoremstyle{remark}
\newtheorem{remark}[proposition]{Remark}

\newcommand{\nc}{\newcommand}
\nc{\I}{{\mathbf 1}}

\nc{\bN}{{\mathbf N}}
\nc{\bM}{{\mathbf M}}
\nc{\cB}{{\mathcal B}}
\nc{\cL}{{\mathcal L}}
\nc{\R}{{\mathbb R}}
\nc{\N}{{\mathbb N}}
\nc{\Z}{{\mathbb Z}}
\nc{\md}{\mathrm{d}}

\DeclareMathOperator{\card}{card}

\nc{\BP}{\mathbb{P}}
\nc{\BE}{\mathbb{E}}
\nc{\BQ}{\mathbb{Q}}
\DeclareMathOperator{\BV}{{\mathbb Var}}
\newcommand{\COV}[1]{\mathsf{Cov}\left( #1 \right)}
\newcommand{\1}{\textbf{1}}

\numberwithin{equation}{section}

\begin{document} 

\renewcommand{\thefootnote}{\fnsymbol{footnote}}
\author{M.A. Klatt\footnotemark[1]\,, G. Last\footnotemark[2]\, and
  D. Yogeshwaran\footnotemark[3]}
\footnotetext[1]{michael.klatt@kit.edu, Karlsruhe Institute of
  Technology, Institute for Stochastics, 76131 Karlsruhe, Germany. }
\footnotetext[2]{guenter.last@kit.edu, Karlsruhe Institute of
  Technology, Institute for Stochastics, 76131 Karlsruhe, Germany. }
\footnotetext[3]{d.yogesh@isibang.ac.in, Theoretical Statistics and
  Mathematics Unit, Indian Statistical Institute, Bangalore, India. }

\title{Hyperuniform and rigid stable matchings} 
\date{\today}
\maketitle

\begin{abstract} 
\noindent 
We study a stable partial matching $\tau$ of {the
$d$-dimensional} lattice with a stationary determinantal point process
$\Psi$ on $\R^d$ with intensity $\alpha>1$. For instance, $\Psi$ might
be a Poisson process. The matched points from $\Psi$ form a stationary
and ergodic (under lattice shifts) point process $\Psi^\tau$ with intensity $1$ that very
much resembles $\Psi$ for $\alpha$ close to $1$. On the other hand
$\Psi^\tau$ is hyperuniform and number rigid, quite in contrast to a
Poisson process. 
We deduce these properties by proving more general
results for a stationary point process $\Psi$, whose so-called
matching flower (a stopping set determining the matching partner of a lattice point)
has a certain subexponential tail behaviour.
For hyperuniformity, we also additionally need to
assume some mixing condition on $\Psi$. Further, if $\Psi$ is a
Poisson process then $\Psi^\tau$ has an exponentially decreasing
truncated pair correlation function.
\end{abstract}


\vspace{0.1cm}

\section{Introduction}\label{intro}

With respect to the degree of order and disorder, the completely
independent Poisson point process (ideal gas) is the exact opposite of a
(standard) lattice with a perfect short- and long-range order. Here we
match these two extremes in the stable sense of Gale and
Shapley~\cite{GaleShapley1962} and Holroyd, Pemantle, Peres and
Schramm~\cite{HPPS09}, where lattice and Poisson points prefer to be
close to each other. Assuming the intensity of the Poisson process to
be larger than one, the matched Poisson points form a point process (a
stable thinning) that inherits properties from both the lattice and the
Poisson process. In fact, if the Poisson  intensity approaches unity,
the thinning becomes almost indistinguishable from a Poisson process in
any finite observation window, while its large-scale density fluctuations
remain anomalously suppressed similar to the lattice;
see the supplementary
video\footnote[4]{\url{https://arxiv.org/src/1810.00265v1/anc/KLY-2018-Supplementary-Video.mp4}}.
In this article, we study such properties of stable partial matchings
between the lattice and a stationary point process.
We shall now give more details.

\begin{figure}[t]
  \centering
  \includegraphics[width=\textwidth]{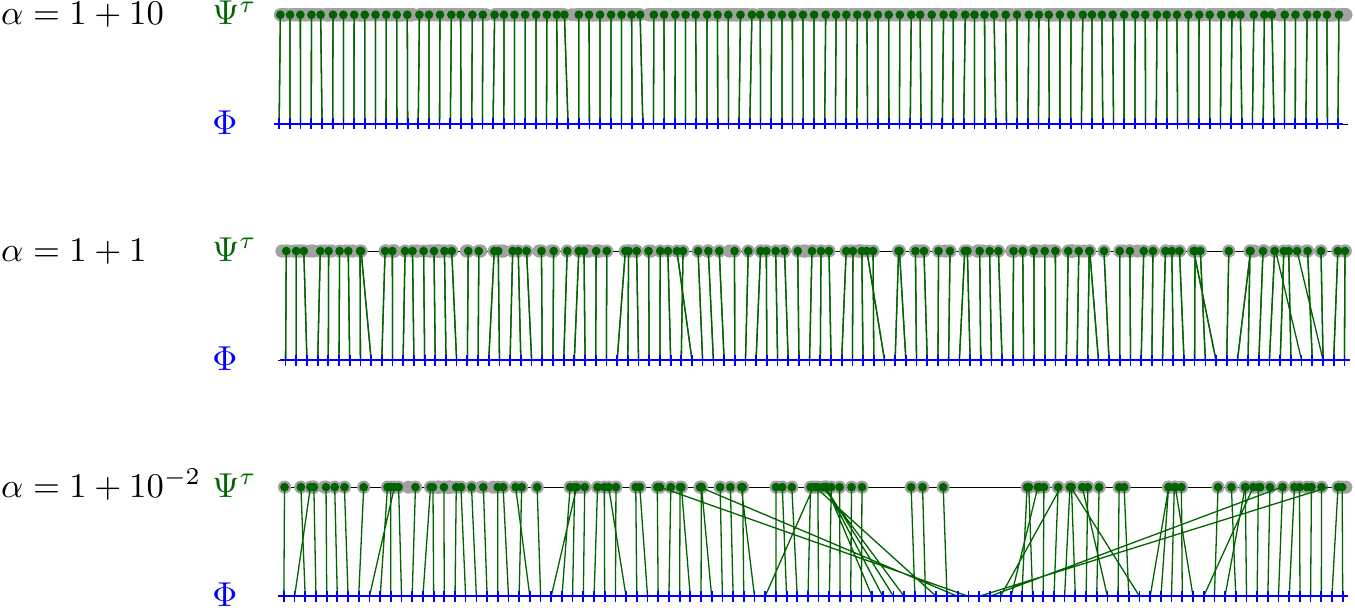}
  \caption{Samples of 1D hyperuniform stable matchings: a stationarized lattice
    $\Phi$ (blue crosses) is matched to a Poisson point process $\Psi$ with
    intensity $\alpha>1$ (gray points) resulting in the hyperuniform point
    process $\Psi^\tau$ (green points).  The matching is visualized by (green)
    lines. Three samples are depicted, each at a different intensity $\alpha$.}
  \label{fig:samples-1D}
\end{figure}

\begin{figure}[t]
  \centering
  \includegraphics[width=\textwidth]{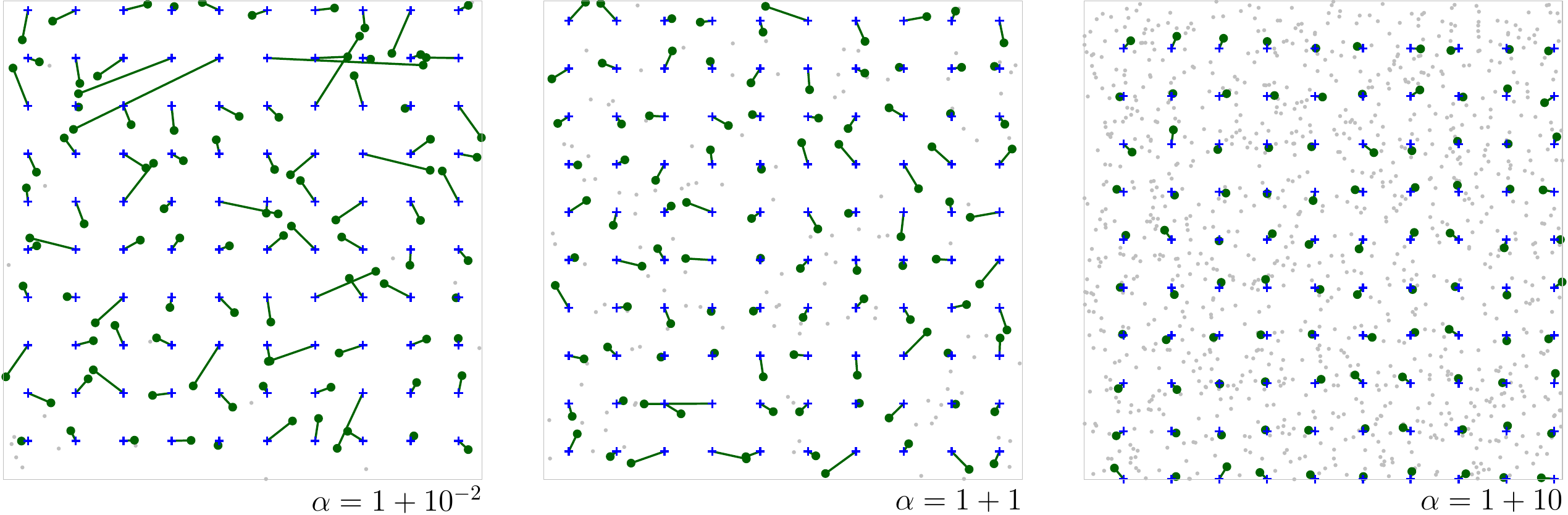}
  \caption{Samples of 2D hyperuniform stable matchings: a stationarized lattice
    $\Phi$ (blue crosses) is matched to a Poisson point process $\Psi$ with
    intensity $\alpha>1$ (gray points) resulting in the hyperuniform point
    process $\Psi^\tau$ (green points).  The matching is visualized by (green)
    lines. Three samples are depicted, each at a different intensity $\alpha$.}
  \label{fig:samples-2D}
\end{figure}

Let $\Psi$ be an ergodic simple point process
on $\R^d$ with finite intensity $\alpha>1$. 
Throughout we identify a {\em point process}
with its {\em support}, where we refer to \cite{LastPenrose17}
for notation and terminology from point process theory.
Let $U$ be a $[0,1)^d$-valued random variable, independent
of $\Psi$ and let $\Phi:=\Z^d+U$ be a {\em randomized lattice}, where 
$B+x:=\{y+x:y\in B\}$ for $x\in\R^d$ and $B\subset\R^d$.
{Further, if $U$ is uniform in $[0,1)^d$, we call the randomized
lattice $\Phi$ a {\em stationarized lattice}.}
A {\em matching} between $\Phi$ and $\Psi$ is a mapping
$\tau\equiv\tau(\Phi,\Psi,\cdot)\colon\R^d\to \R^d\cup\{\infty\}$ 
such that almost surely $\tau(\Phi)\subset\Psi$, $\tau(\Psi)\subset\Phi\cup\{\infty\}$, and
$\tau(x)=p$ if and only if $\tau(p)=x$ for $p\in\Phi$ and $x\in\Psi$.
Since $\alpha>1$, as a rule we will have $\tau(x)=\infty$ for
infinitely many $x\in\Psi$. Following \cite{HPPS09},
we call a matching {\em stable} if there is no pair
$(p,x)\in\Phi\times\Psi$ such that
\begin{align}\label{stablem}
|p-x| < \min\{|p-\tau(p)|,|x-\tau(x)|\},
\end{align} 
{where $|x- \infty| = \infty$ for each $x \in \R^d$.}
(Such a pair is  called \textit{unstable}.)
In the absence of {\em infinite descending chains}
the {\em mutual nearest neighbour matching} shows that
a stable matching exists and is almost surely uniquely 
determined; see~\cite{HPPS09} and Section \ref{secmatchingpp}.
In this paper we study the point process
$\Psi^\tau:=\{\tau(p):p\in\Phi\}=\{x\in\Psi:\tau(x)\ne\infty\}$; a stable thinning of $\Psi$.
Figures~\ref{fig:samples-1D} and \ref{fig:samples-2D} 
show realizations of $\Psi^\tau$ when $\Psi$ is a stationary Poisson point process in 1D and 2D, respectively.

It is not hard to see that $\Psi^\tau$ is stationary and ergodic under translations from $\Z^d$. 
Moreover, if $U$ has the uniform distribution, then $\Psi^\tau$ is stationary and ergodic (under all translations)
and has intensity $1$. It might be helpful to interpret $\Psi^\tau$ as the output process
of a spatial {\em queueing system}, where $\Phi$ represents the
locations {(arrival times)} of customers and $\Psi$ the potential {departure} times.
In fact, a one-sided version of the stable matching can be interpreted as a queueing system with deterministic
arrival times, one server with infinite waiting capacity and a {\em last in, first out} service discipline. The output process of a closely related
system (with a first in, first out service discipline) was studied in \cite{GoldsteinLebowitzSpeer2006,HMNW11}; see~also Section \ref{secone-sided}.

In this paper, we show that the stable thinning $\Psi^\tau$ has several remarkable properties. As mentioned before, when $\Psi$ is the
stationary Poisson point process, $\Psi^{\tau}$ approaches the lattice and the Poisson point process with unit intensity at the two extreme
limits $\alpha \to \infty$ and $\alpha \to 1$, respectively. On the other hand, we prove that if $\Psi$ is a stationary determinantal
point process with a suitably fast decaying kernel, $\Psi^\tau$ is {\em hyperuniform} (or {\em superhomogeneous}), that is, density
fluctuations on large scales are anomalously suppressed; see~\cite{TorquatoStillinger2003, Torquato2018}. This also includes the
stationary Poisson point process.
More precisely, hyperuniformity means that
\begin{align}\label{ehyperun}
\lim_{r\to\infty}\frac{\BV\Psi^\tau(rW)}{\lambda_d(rW)}=0,
\end{align}
where $\lambda_d$ denotes Lebesgue measure on $\R^d$ and $W$ is
an arbitrary convex and compact set with $\lambda_d(W)>0$.
{In general, this notion depends on the observation window
  $W$, see~\cite{KimTorquato2017}.
In our paper, however, we consider the stronger notion of
hyperuniformity defined above.}
We also show that $\Psi^\tau$ is {\em number rigid} when $\Psi$ is a
stationary determinantal point process.
This means that for each bounded Borel set $B\subset\R^d$ the random
number $\Psi^\tau(B)$ is almost surely determined by $\Psi^\tau\cap
B^c$, the restriction of $\Psi^\tau$ to the complement of $B$;
see~\cite{GhoshLebowitz17a,GhoshPeres17} for a definition and discussion of this concept.
Despite these unique properties on large scales, the {\em truncated (or
total) pair correlation function} of  $\Psi^\tau$ is exponentially
decaying (if $U$ is deterministic) when $\Psi$ is the stationary Poisson
point process. In one and two dimensions, there is a close relationship
between hyperuniformity and number rigidity; see~\cite{GhoshLebowitz17b}.
To reiterate why hyperuniformity and number rigidity of $\Psi^{\tau}$ is
interesting, note that the original point process $\Psi$ in many of the
above cases (for example, the stationary Poisson point process) does not need to
be hyperuniform or number rigid.
Further, our results supply a large class of examples of hyperuniform
and number rigid point processes in all dimensions.
There are few examples of number rigid point processes in higher
dimensions ($d \geq 3$).
The ones we are aware of are
small i.i.d.~Gaussian perturbations of a lattice~\cite{Peres2014},
stationary point processes satisfying  DLR (Dobrushin-Landford-Ruelle)
equations with appropriate interacting
potentials~\cite{Dereudre2018} and the hierarchical Coulomb gas in
$d = 3$~\cite{Chatterjee2017}.
We wish to also point out that first rigidity and hyperuniformity
reveal something intrinsically interesting about the point process,
and second
these properties are also useful to understand percolation models on
point processes (see page~5 of \cite{Ghosh2015}).

The paper is organized as follows. We define the matching algorithm
in Section \ref{sdefinition} and strongly supported by \cite{HPPS09}
we show that the stable partial matching between two locally finite
point sets is well defined and unique in the absence of 
infinite  descending chains. In Section \ref{secmatchingpp}, we discuss a few
basic properties of stable partial matchings $\tau$ for a general
stationary and ergodic point process $\Psi$ with intensity
$\alpha\ge 1$. The point process $\Phi$ is assumed to be stationary
with intensity $1$ or a randomized lattice. Theorems \ref{2.8} and
\ref{th_no_inf_chains} show that the matching $\tau$ and hence the
thinned process $\Psi^\tau$ are well-defined for a huge class of point
processes.  In Section \ref{sec:matching_flower}, we introduce and
study the {\em matching flower}, one of our key technical tools.  It
is a {\em stopping set} that determines the matching partner
of a given point. Starting from
Section \ref{stail}, we assume that $\Phi$ is a (possibly randomized)
lattice. In Section \ref{stail}, we consider the tail properties of
the matching distance $\|\tau(0)\|$ of the origin $0$ from its
matching partner and also that of the size of the matching
flower. When $\Psi$ is a stationary determinantal point process (which includes
the Poisson process) with intensity $\alpha > 1$, the first main result of our
paper (Theorem \ref{t4.1}) shows that the distance $\|\tau(0)\|$ has an
exponentially decaying tail; a crucial fact for all of our later results on
determinantal processes. Our proof is inspired by ideas from \cite{HPPS09}.
{The assumption of a determinantal point process allows us
  to use the concentration inequality in \cite{Pemantle14} for
  suitable Lipschitz functionals.}
Then, we show in Theorem \ref{t5.19} that the size
of the matching flower for a stationary point process (satisfying 
the assumptions of Theorem \ref{th_no_inf_chains})
has a sub-exponentially decaying tail provided the matching distance $\|\tau(0)\|$ has a
sub-exponentially decaying tail. Combining Theorems \ref{t4.1} and
\ref{t5.19}, we deduce sub-exponentially decaying tail for the size of
the matching flower of stationary determinantal point process with
intensity $\alpha > 1$ in Corollary \ref{c5.21}. We also study the
tail behaviour of the matching distance seen from an extra point added
to $\Psi$. In case of the Poisson point process, we study the tail
behaviour of the matching distance and the size of the matching flower
of a typical point (Theorem \ref{t4.7} and Proposition \ref{p4.5}). In
Section \ref{sec:hyperuniformity}, we first prove a general criterion
for hyperuniformity (Theorem \ref{t:hypuniform_mixing}), which is
satisfied for $\Psi^\tau$, whenever $\Phi$ is a deterministic lattice
and $\Psi$ is determinantal with a `fast decaying' kernel and
intensity $\alpha > 1$. The general criteria involve
sub-exponentially decaying tail for the size of the matching flower
and suitable decay of mixing coefficients of the point process
$\Psi$.
{Here, we use mixing properties of determinantal point processes.}
In Section \ref{srigid}, we prove number rigidity of $\Psi^\tau$ assuming only
sub-exponentially decaying tail for the size of the matching flower.
In Section \ref{spair}, we assume that $\Phi$ is a deterministic
lattice and that $\Psi$ is a stationary Poisson process and prove that the
truncated pair correlation function is exponentially decaying. 
{In} Section \ref{secone-sided}, we consider a one-sided stable matching on
the line, a simpler version of the two-sided case. This can be
interpreted as a queueing system with a Last-In-First-Out rule and
where the input and departure processes are both considered to be
$\Z$-stationary point processes.
Under weak assumptions, we show that the asymptotic variance profile of
the output process is the same as that of the input process.
In other words, the output process is hyperuniform or non-hyperuniform
or hyperfluctuating, whenever the input is.
In Section~\ref{sec_simulations}, extensive simulations of a matching on
the torus (in 1D, 2D, and 3D) between the lattice and determinantal and
Poisson point processes confirm the hyperuniformity of $\Psi^\tau$ and
study how local and global structural characteristics change with
varying intensity $\alpha$. Our code is freely available as an \verb+R+-package via~\cite{Klatt2018}. We conclude the paper by presenting some conjectures and some further
directions of research.
The appendix contains some basic material on point processes.

{In a sense, the Poisson point process is a degenerate case of a determinantal
  point process.
  Nevertheless, we consider the Poisson point
  process to be a determinantal point process for ease of stating
  our results and proofs.
  The facts about determinantal point processes that we use in our paper
  are provided in the appendix (Section
  \ref{s:app}), see Theorems \ref{t:dpp_conc} and \ref{t:dpp_mixing} as
  well as \eqref{e:Had}.}

{Our results on hyperuniformity and number rigidity are stated under
  general conditions on mixing of the point process $\Psi$ (satisfied by determinantal point processes with fast-decaying kernels) and
  sub-exponentially decaying tail bound on the size of the matching
  flower. 
  As pointed out above, the proof of the latter uses the
  concentration properties of a determinantal process.
  Though exponential decay of
  pair correlation function is proven only for $\Psi$ being a Poisson
  point process, we believe that this can also be extended to suitably
  mixing point process $\Psi$ and again assuming that the size of the
  matching flower has a sub-exponentially decaying tail.}

\section{Stable matchings}\label{sdefinition}

We let $\mathbf{N}$ denote the space of all locally finite subsets
$\varphi\subset\R^d$ equipped with the $\sigma$-field $\mathcal{N}$ generated by
the mappings $\varphi\mapsto \varphi(B):=\card(\varphi\cap B)$,
$B\in\mathcal B^d$. Here $\mathcal B^d$ denotes the Borel $\sigma$-field
in $\R^d$. 
We will identify $\varphi\subset\R^d$ with the associated counting measure. 

Let $d(x,A):=\inf\{\|y-x\|:y\in A\}$ denote the distance between
a point $x\in\R^d$ and a set $A\subset\R^d$, where $\inf\emptyset:=\infty$
and where $\|\cdot\|$ denotes the Euclidean norm.
Let $\varphi\in\bN$ and $x\in\R^d$.
We call $p\in\varphi$ {\em nearest neighbour} of $x$ in $\varphi$ if 
$\|x-p\|\le \|x-q\|$ for all $q\in\varphi$. The lexicographically smallest
among the nearest neighbours of $x$ is denoted by $N^-(\varphi,x)$
and the largest by $N^+(\varphi,x)$.
For completeness we define $N(\varphi,x):=\infty$ if $\varphi=\emptyset$.

Given $\varphi,\psi\in\bN$ we now define a 
\textit{mutual nearest neighbor matching} from  $\varphi$ to $\psi$,
{closely following \cite{HPPS09}.}
For $x\in\R^d$, we define $\tau_1(\varphi,\psi,x):=y$ if
$N^-(\psi,x)=y$ and $N^+(\varphi,y)=x$ or $N^+(\varphi,x)=y$ and $N^-(\psi,y)=x$.
Otherwise we put $\tau_1(\varphi,\psi,x):=\infty$. 
{For all $n\in\N$, we inductively define a mapping
$\tau_{n+1}(\varphi,\psi,\cdot):\R^d\to\R^d$ by
$\tau_{n+1}(\varphi,\psi,x):=\tau_1(\varphi_n,\psi_n,x)$, where
\[
  \varphi_n:=\{x\in\varphi:\tau_n(\varphi,\psi,x)\notin\psi\},\qquad \psi_n:=\{x\in\psi:\tau_n(\varphi,\psi,x)\notin\varphi\}.
\]
}For $x\in\R^d$ we define $\tau(\varphi,\psi,x):=\tau_n(\varphi,\psi,x)$ 
if $x\in(\varphi_n\setminus \varphi_{n+1})\cup (\psi_n\setminus \psi_{n+1})$
for some $n\ge 0$, where $\varphi_0:=\varphi$ and $\psi_0:=\psi$.
Define the sets of unmatched points
\[
  \varphi_\infty:=\bigcap_{n=1}^\infty \varphi_n,\qquad \psi_\infty:=\bigcap_{n=1}^\infty \psi_n.
\]
For $x\in\varphi_{\infty}\cup\psi_{\infty}$, we set $\tau(\varphi,\psi,x):=\infty$.
For completeness we define $\tau(\varphi,\psi,x):=x$ in all other cases, that is for $x\notin\varphi\cup\psi$. 
We refer to the above iterative
procedure for defining the matching as {\em mutual nearest neighbour matching} as 
above or more simply as {\em matching algorithm}. {See Figure
  \ref{fig:matching-algorithm} for an illustration of the algorithm
  for point processes on the line.}

Note that $\tau$ is a measurable mapping on $\bN\times\bN\times\R^d$ with
the {\em covariance property}
\begin{align}\label{eshiftcov}
\tau(\varphi+z,\psi+z,x+z)=\tau(\varphi,\psi,x)+z,\quad \varphi,\psi\in\bN,\,x,z\in\R^d. 
\end{align}

Given points $p\in\varphi$ and $x,y\in\psi$,
$p$ \textit{prefers} $x$ to $y$ if either $\|x-p\|<\|y-p\|$
or $\|x-p\|=\|y-p\|$ and $x$ is lexicographically strictly smaller than $y$.
Correspondingly for $x\in\psi$ and $p,q\in\varphi$,
$x$ \textit{prefers} $p$ to $q$ if either $\|p-x\|<\|q-x\|$
or $\|p-x\|=\|q-x\|$ and $p$ is lexicographically greater than $q$.

\begin{figure}[t]
  \centering
  \includegraphics[width=\textwidth]{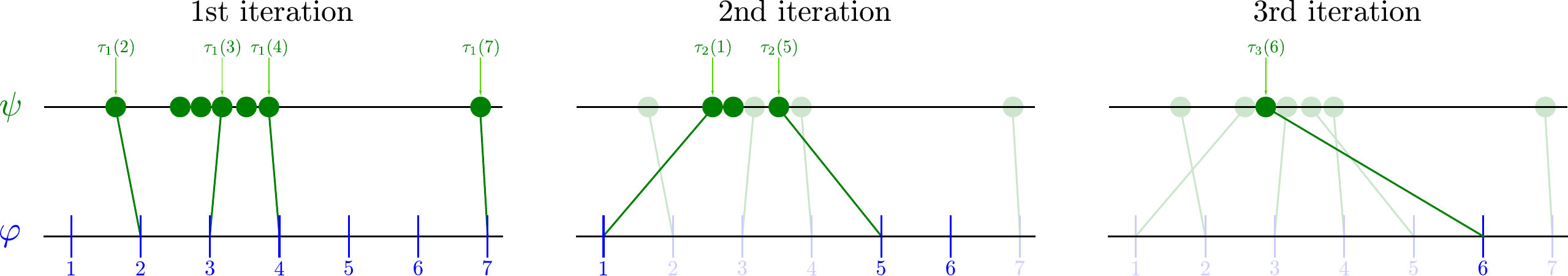}
  \caption{The iterated mutually closest matching algorithm:
  match mutually nearest neighbors (left),
  then repeat considering only the remaining points (center and right).}
  \label{fig:matching-algorithm}
\end{figure}

We now refine the notion of stable matchings as discussed in the introduction.
A (partial) matching is in general defined by a measurable mapping
$\tau':\bN\times\bN\times\R^d\to\R^d\cup\{\infty\}$
such that $\tau'(\varphi,\psi,q)\in\psi\cup\{\infty\}$, 
$\tau'(\varphi,\psi,x)\in\varphi\cup\{\infty\}$, 
and $\tau'(\varphi,\psi,q)=x$ if and only if
$\tau'(\varphi,\psi,x)=q$ for $q\in\varphi$ and $x\in\psi$. {It is
called partial because not all points in $\varphi \cup \psi$ need to be
matched, i.e., some of them are mapped to $\infty$.} Given $\varphi,\psi\in\bN$, we call a pair $(p,x)\in\varphi\times\psi$
unstable in $(\varphi,\psi)$ if $x$ prefers $p$ to
$\tau'(\varphi,\psi,x)$ and $p$ prefers $x$ to $\tau'(\varphi,\psi,p)$.
If there is no unstable pair in $(\varphi,\psi)$, then we call $\tau'$ a
\textit{stable} matching from $\varphi$ to $\psi$.  Moreover if $\tau'$
is a stable matching from $\varphi$ to $\psi$ for all
$(\varphi,\psi)\in\bN\times\bN$, then we simply call it \textit{stable}.

\begin{proposition}\label{stable-matching}
  The mutual nearest neighbor matching $\tau$ is stable.
\end{proposition}
{\em Proof:} Assume that there is an unstable pair $(p,x)\in\varphi\times\psi$.
By definition of $\tau$ there exists an $n\in\N_0$ such that $p\in\varphi_n$ and $x\in\psi_n$
as well as $N^-(\psi_n,p)=\tau(\varphi,\psi,p)\neq x$
or $N^+(\varphi_n,x)=\tau(\varphi,\psi,x)\neq q$, a contradiction.\qed

\bigskip
We call $(\varphi,\psi)\in\bN\times\bN$ \textit{non-equidistant}
if there do not exist $p,q\in\varphi$ and $x,y\in\psi$ with $\{p,x\}\neq\{q,y\}$ and $\|p-x\|=\|q-y\|$.
Note that the mutual nearest neighbor matching is well defined and stable independent of this property,
but if $(\varphi,\psi)$ is \textit{non-equidistant},
then $\tau(\varphi,\psi,\cdot)=\tau(\psi,\varphi,\cdot)$
due to the symmetry of the Euclidean norm.

Let $\varphi,\psi\in\mathbf{N}$ and $n\ge 2$. A sequence $(z_1,\ldots,z_n)$ of
points in $\R^d$ is called a {\em descending chain} in $(\varphi,\psi)$ if
$z_i\in\varphi$ for odd $i\in\{1,\ldots,n\}$, $z_i\in\psi$ for even
$i\in\{1,\ldots,n\}$ and $z_i$ prefers $z_{i+1}$ to $z_{i-1}$ for each
$i\in\{2,\ldots,n-1\}$.
Note that any pair $(z_1,z_2)\in\varphi\times\psi$ is a descending chain.
An infinite sequence $(z_1,z_2,\ldots)$ of points in $\R^d$ is called an
{\em infinite descending chain} if for every $n\ge 2$, $(z_1,\ldots,z_n)$ is a descending chain.
The following result  can be proved as Lemma~15 in \cite{HPPS09}.

\begin{lemma}\label{as-matching-of-points}
  Let $\varphi,\psi\in\bN$,
  and suppose that there is no infinite descending chain in $(\varphi,\psi)$.
  Then $\varphi_{\infty}=\emptyset$ or $\psi_{\infty}=\emptyset$,
  and there is a unique stable partial matching from $\varphi$ to $\psi$,
  which is produced by the mutual nearest neighbor matching.
\end{lemma}

\begin{remark}
  In $d=1$, there is no infinite descending chain in $(\Z,\psi)$ for any $\psi\in\bN$.
  Otherwise there would be an infinite descending chain in $\Z$.
  The corresponding statement does not hold for any $d>1$.
\end{remark}

\begin{proposition}\label{matching-of-psi-tau}
  Let $\varphi,\psi\in\bN$ and $x\in\R^d$, then
  $\tau(\varphi\setminus\varphi_{\infty},\psi\setminus\psi_{\infty},x)=\tau(\varphi,\psi,x)$.
\end{proposition}
{\em Proof:}
If $x\in\varphi_{\infty}\cup\psi_{\infty}$ (or $x\notin\varphi\cup\psi$), then
$\tau(\varphi\setminus\varphi_{\infty},\psi\setminus\psi_{\infty},x)=\tau(\varphi,\psi,x)=\infty$.
Otherwise $x$ prefers $\tau(\varphi,\psi,x)$ to each point in $\varphi_{\infty}\cup\psi_{\infty}$,
which by the definition of the algorithm concludes the proof.
\qed

\bigskip

For later use we define the {\em the matching status} of a
point as the (measurable) function $M\colon\bN\times\bN\times\R^d$ {given by
\begin{align}\label{e3.7}
M(\varphi,\psi,x):=\I\{\tau(\varphi,\psi,x)\in \varphi \cup \psi\},\quad 
(\varphi,\psi,x)\in \mathbf{N}\times \mathbf{N}\times\R^d.
\end{align}
}

\section{Matching of point processes}\label{secmatchingpp}

A (simple) {\em point process} on $\R^d$ (see e.g.\
\cite{Kallenberg,LastPenrose17}) is a random element $\Psi$ of
$\mathbf{N}$ defined on some probability space
$(\Omega,\mathcal{F},\BP)$. Such a point process is {\em stationary}
if $\Psi$ and $\Psi+x$ have the same distribution for all
$x\in\R^d$. In that case we define the {\em intensity} $\alpha$ of
$\Psi$ by $\alpha:=\BE[\Psi([0,1]^d)]$ and we have that
$\BE[\Psi(B)] = \alpha \lambda_d(B)$ for any Borel set $B$.
Some more point process notions such as the correlation functions are
defined in the appendix, see Section~\ref{s:app}.

In this section we consider point processes $\Phi$ and $\Psi$ on $\R^d$.
Later in the paper we shall assume that $\Psi$ is determinantal
(in particular Poisson) and that $\Phi$ is a (randomized) lattice.
Let $\tau$ be the stable matching introduced in the preceding section and define
\begin{align*}
\Phi^\tau:=\{p\in\Phi:\tau(\Phi,\Psi,p)\ne\infty\},\qquad 
\Psi^\tau:=\{x\in\Psi:\tau(\Phi,\Psi,x)\ne\infty\}.
\end{align*}
If $\Phi$ and $\Psi$ are stationary (resp.~ergodic) then it follows from
the covariance property \eqref{eshiftcov} that
$\Phi^\tau$ and $\Psi^\tau$ are stationary (resp.~ergodic) as well.

Next we show that for jointly stationary
and ergodic point processes, in the process with the lesser intensity
all points are matched (almost surely).

\begin{theorem}\label{2.8} Assume that $\Phi$ and $\Psi$ are jointly
stationary and ergodic with intensities 1 and $\alpha\ge 1$, respectively.
Assume also that $(\Phi,\Psi)$  does almost surely
not contain an infinite descending chain. Then $\BP(\Phi^\tau=\Phi)=1$.
If $\alpha=1$, then also $\BP(\Psi^\tau=\Psi)=1$.
\end{theorem}
{\em Proof:} The proof is similar to the one of Proposition~9 in \cite{HPPS09}.
First, it is not hard to show that $\Phi^\tau$ and $\Psi^\tau$
have the same intensity, that is
\begin{align}\label{e3.17}
\BE \Phi^\tau([0,1)^d)=\BE \Psi^\tau([0,1)^d).
\end{align}
Note that $\{\Phi^\tau \neq \Phi\} = \{\Phi_{\infty} \neq \emptyset\}$
and the same holds true for $\Psi$ as well. By the covariance property
\eqref{eshiftcov},the events $\{\Phi^\tau=\Phi,\Psi^\tau=\Psi\}$,
$\{\Phi^\tau=\Phi,\Psi^\tau\ne \Psi\}$ and $\{\Phi^\tau\ne
\Phi,\Psi^\tau= \Psi\}$ are invariant under joint translations
of $\Phi$ and $\Psi$. By Lemma \ref{as-matching-of-points} and ergodicity,
exactly one of these events has probability one. 
The case $\BP(\Phi^\tau\ne \Phi,\Psi^\tau= \Psi)=1$ contradicts
\eqref{e3.17}, so that only the other two cases remain. This proves
the first assertion. If $\alpha=1$, then $\BP(\Phi^\tau=\Phi,\Psi^\tau\ne \Psi)=1$
also contradicts \eqref{e3.17}, proving the second assertion.
\qed

\bigskip
{From now on, we assume that $\Phi=\Z^d+U$ is a randomized lattice
where $U$ is $[0,1)^d$-valued random variable.
The point process $\Phi$ as well $\Phi^\tau$ and $\Psi^\tau$ are
$\Z^d$-stationary, that is distributionally invariant under shifts from
$\Z^d$.
However, $\Phi$ is generally not stationary unless $U$ is uniformly
distributed on $[0,1)^d$.} 
Given $u\in\R^d$, we often write $\Phi_u:=\Z^d+u$ and
$\Psi^{\tau}_u:=\{x\in\Psi:\tau(\Phi_u,\Psi,x)\ne\infty\}$.

\begin{theorem}\label{2.11}
Assume that $\Psi$ is a stationary point process with intensity
$\alpha\ge 1$. Assume also that $\Psi$ is ergodic under translations
from $\Z^d$. Let $\Phi:=\Z^d+U$, where $U$ is a $[0,1)^d$-valued
random variable, independent of $\Psi$.
Then assuming that $(\Phi,\Psi)$ does almost surely not contain an
infinite descending chain, the assertions of  Theorem \ref{2.8} hold.
\end{theorem}
{\em Proof:}  Assume first that $U$ is uniformly
distributed, so that $\Phi$ is stationary. Let $A\subset\mathbf{N}\times\mathbf{N}$ be
measurable and invariant under (diagonal) translations.
Then
\begin{align*}
\BP((\Phi,\Psi)\in A)=\BP((\Z^d+U,\Psi)\in A)
=\BP((\Z^d,\Psi-U)\in A)=\BP((\Z^d,\Psi)\in A),
\end{align*}
where we have used the assumed independence and stationarity of $\Psi$.
Since the event $\{(\Z^d,\Psi)\in A\}$ is invariant under
(diagonal) translations
of $\Psi$ by elements of $\Z^d$, we obtain that 
$\BP((\Z^d,\Psi)\in A)\in\{0,1\}$ and hence the joint ergodicity
of $(\Phi,\Psi)$. Moreover, 
$(\Phi,\Psi)$ does almost surely not contain an infinite descending chain,
so that Theorem \ref{2.8} applies.

Let $u\in[0,1)^d$.
By the covariance property \eqref{eshiftcov}, we have
$\Phi^\tau_u:=(\Phi_u)^\tau\overset{d}{=}\Phi^\tau_0+u$ and $\Psi^\tau_u\overset{d}{=}\Psi^\tau_0+u$.
Hence $\BP(\Phi^\tau_u=\Phi_u)$ does not depend on $u$.
By the first step of the proof we have
$\int_{[0,1)^d}\BP(\Phi^\tau_u=\Phi_u)\,du=1$, and therefore
$\BP(\Phi^\tau_u=\Phi_u)=1$ for all $u\in[0,1)^d$.
If $\alpha=1$, then $\int_{[0,1)^d}\BP(\Psi^\tau_u=\Psi_u)\,du=1$, and therefore
$\BP(\Psi^\tau_u=\Psi_u)=1$ for all $u\in[0,1)^d$.
This implies the assertion for general $U$.\qed

\begin{remark}\label{r3.3}\rm
It is easy to see that a \textit{mixing} point process $\Psi$ has the ergodicity property
assumed in Theorem \ref{2.11}.
Our main example, the stationary determinantal point processes are \textit{mixing};
see Theorem 7 in \cite{Soshnikov00} and also Theorem \ref{t:dpp_mixing} in the 
appendix for more quantitative bounds.
\end{remark}

\begin{theorem}\label{th_no_inf_chains}
Assume that $\Phi$ is a randomized lattice,
$\Psi$ is a stationary point process,
and that for each $n\in\N$, the $n$-th correlation function
of $\Psi$ exists and is bounded by $n^{\theta n}c^n$ for some $c>0$ and $\theta \in [0,1)$.
Then $(\Phi,\Psi)$ does almost surely not contain an
infinite descending chain.
\end{theorem}
{\em Proof:} The proof is a direct consequence of the following lemma, which will also be needed later.
\qed

\bigskip
We denote by $\kappa_d$ the volume of the unit ball in $\R^d$.

\begin{lemma}\label{descending_chains}
Let the assumptions of Theorem~\ref{th_no_inf_chains} be satisfied.
Let $q_0\in\Z^d$, $n\in\N$ and $b>0$. Let $A_{n,b}$ be the event
that there is a descending chain $(q_0+U,x_1,q_1,\ldots,x_n,q_n)$
in $(\Phi,\Psi)$ such that $\|x_1-q_0-U\|\le b$. Then
$\BP(A_{n,b})\le a^n_n/n!$, where $a_n:=n^{\theta}c\kappa_d^2(b^2+2b\sqrt{d})^d$.
\end{lemma}
{\em Proof:} To simplify the notation, we first assume that $U\equiv 0$.
By stationarity, we can further assume that $q_0=0$. 
The indicator function of $A_{n,b}$ is bounded by
\begin{align*}
\sideset{}{^{\ne}}\sum_{x_1,\ldots,x_n\in\Psi}\sum_{q_1,\ldots,q_n\in\Z^d}
\I\{b\ge \|x_1\|\ge \|q_1-x_1\|\ge\cdots\ge \|x_n-q_{n-1}\|\ge \|q_n-x_{n}\|\}. 
\end{align*}
Taking expectations and using our assumptions on $\Psi$ gives
\begin{align*}
\BP(A_{n,b})\le n^{\theta n}c^n\int\sum_{q_1,\ldots,q_n\in\Z^d}
\I\{b\ge \|x_1\|\ge \|q_1-x_1\|\ge\cdots\ge &\|x_n-q_{n-1}\|\ge \|q_n-x_{n}\|\}\\
&d(x_1,\ldots,x_n). 
\end{align*}
Letting $W:=[0,1)^d$ this can be written as
\begin{align*}
\BP(A_{n,b})\le n^{\theta n}c^n\int_{W^n}\sum_{p_1,\ldots,p_n\in\Z^d}&\sum_{q_1,\ldots,q_n\in\Z^d}
\I\{b\ge \|p_1+u_1\|\ge \|q_1-p_1-u_1\|\ge \\
\\ &\cdots\ge \|p_n+u_n-q_{n-1}\|\ge \|q_n-p_{n}-u_n\|\}\,
d(u_1,\ldots,u_n).
\end{align*}
By a change of variables the right-hand side equals
\begin{align*}
n^{\theta n}c^n\int_{W^n}\sum_{p_1,\ldots,p_n\in\Z^d}\sum_{q_1,\ldots,q_n\in\Z^d}
\I\{b\ge &\|p_1+u_1\|\ge \|q_1-u_1\|\ge \\
&\cdots\ge \|p_n+u_n\|\ge \|q_n-u_n\|\}d(u_1,\ldots,u_n).
\end{align*}
Since $\|q_j\| \le \|q_j-u_j\|+\sqrt{d}$ for $j\in\{1,\ldots,n\}$,
this can be bounded by 
\begin{align*}
n^{\theta n}c^n(\kappa_d(b+2\sqrt{d})^d)^n&\int_{W^n}\sum_{p_1,\ldots,p_n\in\Z^d}
\I\{b\ge \|p_1+u_1\|\ge\cdots\ge \|p_n+u_n\|\}\,d(u_1,\ldots,u_n)\\
&=n^{\theta n}c^n(\kappa_d(b+2\sqrt{d})^d)^n\int \I\{b\ge \|y_1\|\ge\cdots\ge \|y_n\|\}\,d(y_1,\ldots,y_n)\\
                  &=\frac{n^{\theta  n}c^n(\kappa_d(b+2\sqrt{d})^d)^n(\kappa_db^d)^n}{n!}.
\end{align*}

In the general case, we can assume that $q_0=U$. We then need
to replace $x_1$ by $x_1-u$, and $q_j$ by $q_j+u$ and integrate $u$ with respect
to the distribution of $U$. The details are left to the reader.\qed

\bigskip
Note that, we have actually proved a bound on the expected number of
descending chains.
The following lemma can be proved as the previous one.

\begin{lemma}\label{descending_chains2}
Let the assumptions of Theorem~\ref{th_no_inf_chains} be satisfied.
Let $x\in\R^d$, $n\in\N$ and $b>0$. Let $A_{n,b}$ be the event
that there is a descending chain $(x,q_1,x_1,\ldots,q_n,x_n)$
in $(\Psi\cup\{x\},\Phi)$ such that $\|q_1-x\|\le b$. Then
the assertion of Lemma \ref{descending_chains} holds.
\end{lemma}

Due to \eqref{e:Had}, we have that stationary determinantal point
processes satisfy the conditions of Theorem \ref{th_no_inf_chains} and 
hence we have the following corollary {via Theorem \ref{2.11}.}
\begin{corollary}\label{c3.6} Assume that $\Phi$ is a randomized lattice
and that $\Psi$ is a stationary determinantal point process (in particular Poisson).
Then $(\Phi,\Psi)$ does almost surely not contain an
infinite descending chain {and thus the assertions of  Theorem \ref{2.8} hold.}
\end{corollary}
\begin{remark}\label{DPP_no_descending_chains}\rm
  There are many more examples of point processes satisfying the
  assumptions of Theorem~\ref{th_no_inf_chains} including
  $\alpha$-determinantal point processes with `fast decaying' kernels
  for $\alpha = -1/m, m \in \N$, rarefied Gibbsian point processes, many
  Cox point processes etc. See (1.12) and Sections 2.2.2 and 2.2.3
  in~\cite{BYY18} for these examples and also see Sections~4 and 7
  in~\cite{Daley05}. 
\end{remark}

\begin{remark}
  Note that the absence of infinite descending chains in
  Theorem~\ref{th_no_inf_chains} is independent of the intensity
  $\alpha$ of $\Psi$ and that there is almost surely no infinite
  descending chain in $(\Psi,\Phi)$ as well as in $(\Phi,\Psi)$.  Even
  for an unmatched point, all descending chains are almost surely finite
  (but there are infinitely many).
\end{remark}

\section{Matching flower}\label{sec:matching_flower}

Here we define the \textit{matching flower} $F(\Phi,\Psi,z)$ of a point
$z\in\Phi\cup\Psi$, which is a random set that determines the matching
partner of $z$.
Importantly, it is a stopping set.
We can define this flower by searching the matching partner of $z$
through iterating over all potential matching partners, competitors and
their matching neighbors.

\begin{figure}[t]
  \centering
  \includegraphics[width=\textwidth]{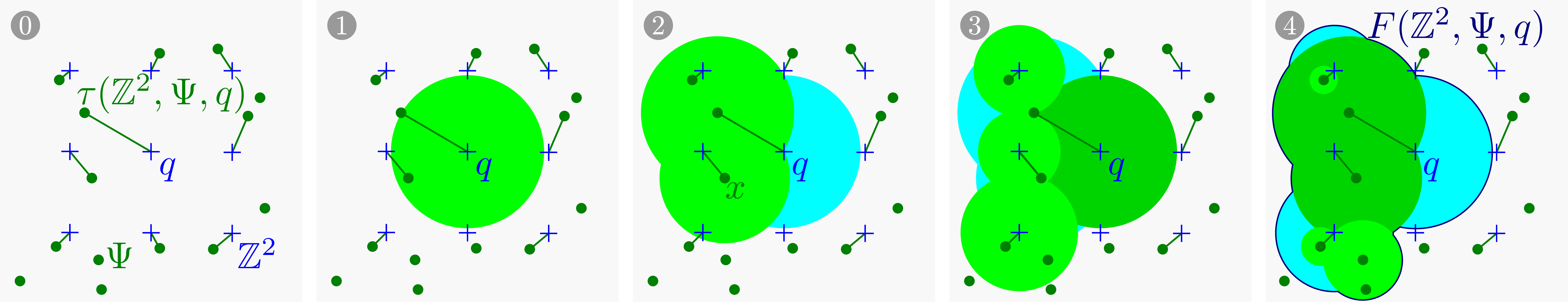}
  \caption{The construction of a matching flower,
    (0)~starts with the matching $\tau$ between $\Z^2$ and $\Psi$.
    (1)~A ball grows at $q$ until it touches the matching partner $\tau(\Z^2,\Psi,q)$ in $\Psi$.
    (2)~Two new balls touching $q$ are centered at $\tau(\Z^2,\Psi,q)$ and $x$ (which is preferred by $q$).
    (3)~New balls grow around the lattice points inside the balls at $\tau(\Z^2,\Psi,q)$ and $x$.
    (4)~To points of $\Psi$ inside these balls, new balls are assigned that touch their
    lattice neighbors.
    Since no new lattice points are accessed, the algorithm terminates.
    The matching flower $F(\Z^2,\Psi,q)$ is the union of all balls from
    (1)--(4).}
  \label{fig:matching-flower}
\end{figure}

Figure~\ref{fig:matching-flower} illustrates an intuitive algorithm to
construct the matching flower {for the example} of a lattice point
$q\in\Z^d$.
First a ball is centered at $q$ that touches its matching partner $\tau(\Z^d,\Psi,q)$. The second step determines which points in $\Psi$ are preferred or equal to $\tau(\Z^d,\Psi,q)$ (by $q$), which have to be inside the ball from
the first step or touch it,
then to each of these points a new ball is assigned that touches $q$.
The third step determines which lattice points are preferred or equal to $q$ (by
these points in $\Psi$), then to each one of them a ball is assigned
that touches the corresponding point in $\Psi$.
Finally, the second and third steps are iterated (considering each new
lattice point instead of $q$) until no new points can be found in any of
the balls.
An alternative representation of the matching flower is via an ordered
\textit{graph of preference}.
We do not formally define it but refer to Fig.~\ref{fig:matching-flower-graph} for an illustration.

If there are no infinite descending chains in $(\Phi,\Psi)$ and if there
is a matching partner $\tau(\Phi,\Psi,z)\neq\infty$ of $z$, then the
matching flower $F(\Phi,\Psi,z)$ is bounded.
This is because $\Psi$ and $\Phi$ are locally finite and thus there
are only finitely many chains whose second element is closer to $z$ than
$\tau(\Phi,\Psi,z)$.
Because of the absence of infinite descending chains, the iterative
search for nearest neighbors terminates in a pair of mutually nearest
neighbors after a finite number of steps.
Note that $F(\Phi,\Psi,q)$ does not only determine the matching
status and partner of $q\in\Phi$ but does so for all points of $\Phi$
that are in the matching flower.

For a rigorous mathematical analysis, we need a more explicit definition
of the matching flower {and we do so now by defining the matching flower on two locally finite sets}. Let $\varphi,\psi\in\mathbf{N}$ and $z\in\varphi\cup\psi$. If
$\tau(\varphi,\psi,z)=\infty$ we set
$F(\varphi,\psi,z):=\R^d$.
Assume now that $\tau(\varphi,\psi,z)\ne \infty$.
We call a descending chain
$c=(z_1,\ldots,z_n)$ in $(\varphi,\psi$) a {\em competing chain}
in $(\varphi,\psi$) for $z$ if $z_1=z$ and 
$z_2$ is equal or preferred to $\tau(\varphi,\psi,z_1)$.
In this case, we define
\begin{align*}
  F_c:=B(z_1,\|z_1-z_2\|)\cup \bigcup^{n}_{i=2}
  B(z_i,\|z_i-z_{i-1}\|){,}
\end{align*}
{where $B(z,r)$ denotes a ball with center $z$ and radius $r$.}
Then the matching flower of $z$ (w.r.t.\ $(\varphi,\psi)$) is defined
by $F(\varphi,\psi,z):=\bigcup F_c$,
where the union is over all competing chains $c$ for $z$.
This definition is similar to (2.1) in \cite{LastPenrose2013}.  
Just for completeness we define $F(\psi,\varphi,z):=\{z\}$
for $z\notin\varphi\cup\psi$.
\begin{remark}
  If $(\varphi,\psi)$ is non-equidistant, then $F(\varphi,\psi,\cdot)=F(\psi,\varphi,\cdot)$.
\end{remark}

For $z\in\R^d$ and $\varphi\in\mathbf{N}$, we define $\varphi^z:=\varphi\cup\{z\}$.
Although the definition of the matching flower involves the global matching,
the mapping $(\varphi,\psi)\mapsto F(\varphi^z,\psi,z)$
has the following useful {\em stopping set} property.
For a set $W\subset\R^d$ we abbreviate $\varphi_W:=\varphi\cap W$
and $\varphi^z_W:=(\varphi^z)_W$.
For $k\in\N$ the {\em $k$-th nearest neighbour} of $z\in\R^d$ in $\varphi$
is defined as the point $N_k(\varphi,z)$ such that the number of points in
$\varphi$ that $z$ prefers to $N_k(\varphi,z)$ is exactly equal to $k-1$.

\begin{lemma}\label{matching-flower-is-stopping-set}
Let $W\subset\R^d$ with $z\in W$
and let $(\varphi,\psi)\in\mathbf{N}\times\mathbf{N}$.
Then $F(\varphi^z,\psi,z)\subset W$ iff
$F(\varphi^z_W,\psi_W,z)\subset W$. In this case
$\tau(\varphi^z,\psi,z)=\tau(\varphi^z_W,\psi_W,z)$.
The corresponding statement also holds for $\psi^z$.
\end{lemma}
{\em Proof:} We can assume that $z\notin \psi$. (Otherwise the
assertion is trivial.)

Assume that $F(\varphi^z,\psi,z)\subset W$ and
$\tau(\varphi^z,\psi,z)\neq\infty$.
Then there exists $k\in\N$ such that
$\tau(\varphi^z,\psi,z)=N_k(\psi,z)$.
Let $c$ be a descending chain in $(\varphi^z,\psi)$ starting in $z_1=z$
and $q_1\in \{N_1(\psi,z),\ldots,N_{k}(\psi,z)\}$.
Since $F(\varphi^z,\psi,z)\subset W$ we obtain that $c$ is also a
descending chain in $(\varphi^z_W,\psi_W)$.
This shows that $N_j(\psi,z)=N_j(\psi_W,z)$ for each $j\in\{1,\ldots,k\}$. 
Moreover, applying the matching algorithm from Section \ref{sdefinition}
to both $(\varphi^z,\psi)$ and $(\varphi^z_W,\psi_W)$ yields that
$\tau(\varphi^z,\psi,z)=\tau(\varphi^z_W,\psi_W,z)$.
But then any competing chain in $(\varphi^z_W,\psi_W)$ for $z$
is also a competing chain in $(\varphi^z,\psi)$ for $z$.
It follows that $F(\varphi^z_W,\psi_W,z)\subset W$.

\begin{figure}[t]
  \centering
  \includegraphics[width=\textwidth]{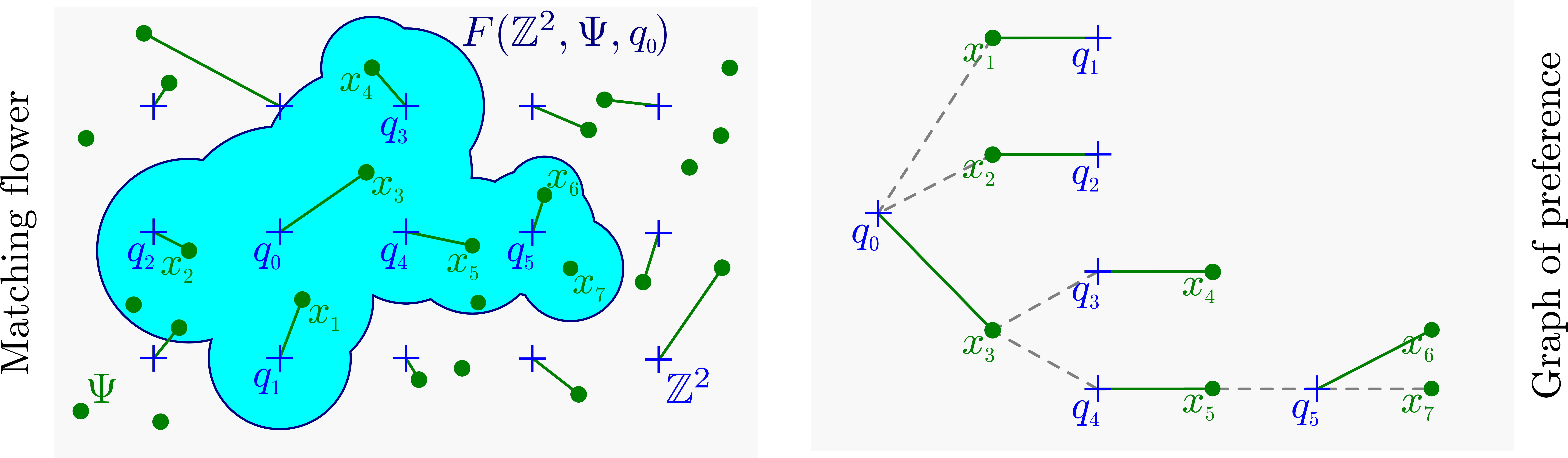}
  \caption{A matching flower $F(\mathbb{Z}^2,\Psi,q_0)$ for a lattice
  point $q_0$ (left) and its ordered graph of preference (right),
  with preferences ordered from top to bottom and from right to
  left (children are preferred to parent points).
  In general, the graph does not have to be a tree.}
  \label{fig:matching-flower-graph}
\end{figure}

Assume that $F(\varphi^z_W,\psi_W,z)\subset W$
and $\tau(\varphi^z_W,\psi_W,z)\ne\infty$.
Thus there exists $k\in\N$ such that $\tau(\varphi^z_W,\psi_W,z)=N_k(\psi_W,z)$.
Since $B(z,\|N_k(\psi_W,z)-z\|)\subset W$, we obtain that
$N_j(\psi_W,z)=N_j(\psi,z)$ for each $j\in\{1,\ldots,k\}$. 
Let $c=(z_1,\ldots,z_n)$ be a descending chain in $(\varphi^z,\psi)$
with $z_1=z$ and $z_2\in \{N_1(\psi,z),\ldots,N_{k}(\psi,z)\}$.
Since the pair $(z_1,z_2)$ is a descending chain in $(\varphi^z_W,\psi_W)$
and $F(\varphi^z_W,\psi_W,z)\subset W$ we obtain that $B(z_2,\|z_2-z_1\|)\subset W$.
Since $\|z_3-z_2\|\le \|z_2-z_1\|$ we obtain that $z_3\in W$, so that
$(z_1,z_2,z_3)$ is a descending chain in $(\varphi^z_W,\psi_W)$.
It follows inductively that $c$ is a descending chain in
$(\varphi^z_W,\psi_W)$.
Applying the matching algorithm from Section \ref{sdefinition}
to both $(\varphi^z_W,\psi_W)$ and
$(\varphi^z,\psi)$ shows that
$\tau(\varphi^z_W,\psi_W,z)=\tau(\varphi^z,\psi,z)$ and in turn
$F(\varphi^z,\psi,z)\subset W$.

In case $\tau(\varphi^z,\psi,z)=\infty$ or $\tau(\varphi^z_W,\psi_W,z)=\infty$, then by the definition of the
matching flower $W=\R^d$ and hence the claim follows trivially. \qed

\begin{proposition}\label{finite-matching-flower}
  Suppose that $\Phi$ is a randomized lattice and that $\Psi$ is a stationary determinantal point process (in particular Poisson).
  Let $z\in\R^d$.
  Then the following implication holds almost surely:
  if $\tau(\Phi,\Psi,z)\ne\infty$ then $F(\Phi,\Psi,z)$ is bounded.
\end{proposition}
{\em Proof:}
Assume that $\tau(\Phi,\Psi,z)\ne \infty$.
Then $F(\Phi,\Psi,z)$ is bounded iff there
are only finitely many competing chains for $z$.
Since there are only finitely many choices for
the second point of such a competing chain,
the only way to have infinitely many of them
is via an infinite descending chain.
However, by Corollary~\ref{c3.6}
such chains do almost surely not exist.
\qed

\section{Tail bounds for the matching distance and flower}\label{stail}

From now on, $\Psi$ shall always be a stationary point process with intensity $\alpha \geq 1$. 
Here we shall prove tail bounds for the distance to the matching partner of a lattice point 
and the matching flower of a lattice point when $\Psi$ is a
determinantal point process (see Section \ref{s:app} for definitions) and we show the same for points of $\Psi$ as
well when $\Psi$ is a Poisson point process. 

\subsection{Lattice points matched to determinantal point processes}
\label{sec:DPP}

If $\Psi$ is determinantal and $\alpha>1$, we derive
an exponential tail bound for the length of the random vector
$\tau(\Z^d,\Psi,0)$, that is for the distance of the origin from its
matching partner, when matching the lattice $\Z^d$ to $\Psi$. 
\begin{theorem}\label{t4.1}
  Assume that $\Psi$ is determinantal (in particular
  Poisson) with intensity $\alpha>1$.
  Then there exists $c_1>0$ such that
  $\BP(\|\tau(\Z^d,\Psi,0)\|>r)\le c_1 e^{-c^{-1}_1r^d}$
  for all $r\ge 0$.
\end{theorem}

Our proof of Theorem \ref{t4.1} was inspired by that of Theorem~21 in~\cite{HoHoPe09}.
First we need to establish a few auxiliary results.
We write $\mu_0$ for the counting measure supported by $\Z^d$, that is
we let $\mu_0(B):=\card(\Z^d\cap B)$ for each Borel set $B\subset\R^d$.
In the following lemma, we use the {\em intrinsic volumes}
$V_0(K),\ldots,V_d(K)$ of a compact convex subset $K$ of $\R^d$;
see e.g.\ Appendix A.3 in \cite{LastPenrose17}. In the following proof and also
later we denote by $B_r$  
the ball with radius $r\ge 0$ centered at the origin.
We shall frequently use the following version of a classical fact; see~\cite{Wills1973}.
For the reader's convenience, we give a short proof.

\begin{lemma}\label{lcounting}
Let $K$ be a compact convex subset of $\R^d$. 
There exist constants $c_0,...,c_{d-1}$ (which only depend on the dimension), such that
\begin{align*}
    |\lambda_d(K)-\mu_0(K)|\le c_0V_0(K)+\cdots+ c_{d-1}V_{d-1}(K).
\end{align*}
\end{lemma}
{\em Proof:} The proof is based on basic convex geometry. Let $W:=[0,1)^d$
and define $I:=\{q\in\Z^d: (W+q)\cap K\ne\emptyset\}$. Then
\begin{align}
\label{e:muineq}
\mu_0(K)\le\card I \le \lambda_d(K+B_{\sqrt{d}}),
\end{align}
where $K+L:=\{x+y:x\in K,y\in L\}$ is the {\em Minkowski sum}
of $K$ and a set $L\subset\R^d$.
By the {\em Steiner formula} (see e.g.\ Appendix A.3 in \cite{LastPenrose17})
\begin{align}
\label{e:steiner}
\lambda_d(K+ B_{\sqrt{d}})=\lambda_d(K)+\sum^{d-1}_{j=0}d^{(d-j)/2}\kappa_{d-j}V_j(K),
\end{align}
where $\kappa_j$ is the volume of the unit ball in $\R^j$.
This gives one of the desired inequalities.

Let $J_1:=\{q\in\Z^d: (W+q)\subset K^0\}$
and
$J_2:=\{q\in\Z^d: (W+q)\cap \partial K\ne\emptyset\}$,
where $K^0$ and $\partial K$ denote the {\em interior} and the {\em boundary} 
of $K$ respectively. Then
\begin{align*}
\lambda_d(K)\le \card J_1+\card J_2\le \mu_0(K)+ \card J_2. 
\end{align*}
For $r>0$ let $\partial_rK:=\{x\in\R^d: d(x,\partial K)\le r\}$.
Since $\cup_{q\in J_2}(W+q)\subset \partial_{\sqrt d}K$, we have that
\begin{align*}
\card J_2\le \lambda_d(\partial_{\sqrt d}K)
\le 2(\lambda_d(K+B_{\sqrt{d}})-\lambda_d(K)),
\end{align*}
where the last inequality comes from (3.19) in \cite{HLS16}.
Applying Steiner's formula again, concludes the proof.
\qed

\bigskip

Unless stated otherwise, we now assume that $\Psi$ satisfies
the assumptions of Theorem~\ref{t4.1}.
Further we abbreviate $\tau(x):=\tau(\Z^d,\Psi,x)$, $x\in\R^d$.

\begin{lemma}\label{l4.5} Let $r>0$ and assume that
$\mu_0(\{p\in B_{2r}:\tau(p)\in B_{r}\})<\Psi(B_r)$. Then
$\|\tau(\Z^d,\Psi,0)\|\le r$.
\end{lemma}
{\em Proof:} By assumption there must exist an $x\in\Psi\cap B_{r}$
such that $\tau(x)\notin B_{2r}$. Assume now on the contrary, that
$\|\tau(\Z^d,\Psi,0)\|=\|\tau(0)\|> r$. But then the pair $(0,x)$ would not be stable, a contradiction.\qed

\bigskip
For a (possibly half-open) cube $B\subset\R^d$ and $q\in \Z^d\cap B$, we say that
$B$ {\em is decisive for} $q$ if the matching flower
$F(\Z^d,\Psi,q)$ is a subset of $B$. In this case we have $\tau(q)\in B$.
Moreover, the event that $B$ is decisive for $q$ does only depend
on $\Psi\cap B$.
Given $m\in\N$ we write $Q(m):=(-m,m]^d$.

\begin{lemma}\label{l4.7} For each $\varepsilon>0$ there exists $m\in\N$ such that
\begin{align*}
\BE \mu_0(\{p\in Q(m):\text{$Q(m)$ is not decisive for $p$}\})\le \varepsilon (2m)^{d}.
\end{align*}
\end{lemma}
{\em Proof:} We combine the proof of Corollary~17 in~\cite{HoHoPe09} with the
properties of the matching flower. Let $n\in\N$. The distribution of the
random set
\begin{align*}
U_n:=\{q\in\Z^d: \text{$Q(n)+q$ is not decisive for $q$}\}
\end{align*}
is invariant under shifts from $\Z^d$. 
By Proposition~\ref{finite-matching-flower}, the matching flower is almost surely finite.
Therefore we may choose an integer $n$ so large such that
\[
  \BP(\text{$Q(n)$ is not decisive for $0$})\le\varepsilon/2.
\]
Therefore we obtain for each $m\ge n$
\begin{align*}
\BE\mu_0(U_n\cap Q(m))=\sum_{q\in Q(m)}\BP(\text{$Q(n)+q$ is not decisive for $q$})
\le(2m)^d\varepsilon/2.
\end{align*}
Since $Q(n)+q\subset Q(m)$ for $q\in Q(m-n)$
and the event $\{\text{$Q(m)$ is not decisive for $q$}\}$
is decreasing in $m$, we obtain that
\begin{align*}
\BE&\mu_0(\{q\in Q(m):\text{$Q(m)$ is not decisive for $q$}\})\\
&\le \mu_0(Q(m)\setminus Q(m-n))+\sum_{q\in Q(m-n)}\BP(\text{$Q(n)+q$ is not decisive for $q$})\\
&\le (2m)^d-(2(m-n))^d+ (2m)^d\varepsilon/2.
\end{align*}
For sufficiently large $m$ the last expression becomes smaller than
$(2m)^d\varepsilon$.\qed

\bigskip

{\em Proof of Theorem \ref{t4.1}:} We let $\varepsilon>0$ be such that
$\varepsilon (2^{d+1}+2)<\alpha-1$ and so, we also have that $\alpha-\varepsilon>1$.
Choose $m$ as in Lemma \ref{l4.7}.
For $q\in\Z^d$ we define $Q_q:=Q(m)+2mq$ and the random variable
\begin{align*}
Y_q:=\mu_0(\{p\in Q_q:\text{$Q_q$ is not decisive for $p$}\}).
\end{align*}

Next we take $r>0$ and define
\begin{align*}
I_r&:=\{q\in\Z^d:Q_q\subset B_{2r}\setminus B_r\},\\
S_r&:=(B_{2r}\setminus B_r)\setminus \bigcup_{q\in I_r}Q_q.
\end{align*}
Choosing $r$ sufficiently large, we can assume that
\begin{align}\label{e4.7}
  \mu_0(S_r)\le \varepsilon \mu_0(B_r).
\end{align}
The above inequality can be derived as follows. For $r$ sufficiently large, we have that
\[ S_r  \subset ( B_{2r} \setminus  B_{2r - 2m{\sqrt{d}}} )  \cup (B_{r + 2m\sqrt{d}} \setminus B_r). \]
Hence, we derive that
\begin{align*}
  \mu_0(S_r) & \leq \lambda_d(S_r + B_{\sqrt{d}}) \leq \lambda_d(B_{2r+\sqrt{d}}) - \lambda_d(B_{2r-3m\sqrt{d}}) + \lambda_d(B_{r+3m\sqrt{d}}) - \lambda_d(B_{r-\sqrt{d}}) \\
  & \leq C r^{d-1} \leq \varepsilon^{1/2} \lambda_d(B_r)  \leq \varepsilon \mu_0(B_r),
\end{align*}
where $C$ is a constant, the first inequality is similar to
\eqref{e:muineq}, the second is from the above-mentioned inclusion for $S_r$, the third
inequality is due to Steiner's formula \eqref{e:steiner} and the fact
that $V_j(B_r) = O(r^j)$ for all $j$. The last inequality is from
Lemma~\ref{lcounting} and the above bound on growth rate of the $V_j$. 

Next we note that
\begin{align*}
\mu_0(\{q\in B_{2r}:\tau (q)\in B_{r}\})&\le \mu_0(B_r)+\mu_0(S_r)
+\mu_0\Big(\Big\{p\in \bigcup_{q\in I_r}Q_q:\tau(p)\in B_{r}\Big\}\Big).
\end{align*}
If $p\in Q_q\subset B_{2r}\setminus B_r$ and $\tau(p)\in B_r$, 
then $Q_q$ is not decisive for $p$. Therefore
\begin{align}\label{e4.13}
\mu_0(\{q\in B_{2r}:\tau (q)\in B_{r}\})
\le \mu_0(B_r)+\mu_0(S_r)+Z_r,
\end{align}
where
\begin{align*}
Z_r:=\sum_{q\in I_r} Y_q.
\end{align*}
Consider the events
\begin{align*}
C_r&:=\{\Psi(B_r)>(\alpha-\varepsilon)\mu_0(B_r)\},\\
D_r&:=\Big\{Z_r<2^{d+1}\varepsilon \mu_0(B_r)\Big\}.
\end{align*}
If the event $D_r$ occurs, then we obtain from \eqref{e4.13},
\eqref{e4.7} and the definition of $\varepsilon$ that
\begin{align*}
\mu_0(\{q\in B_{2r}:\tau (q)\in B_{r}\})&
\le (1+\varepsilon+2^{d+1}\varepsilon) \mu_0(B_r)<(\alpha-\varepsilon)\mu_0(B_r).
\end{align*}
If $C_r\cap D_r$ occurs, we therefore have
\begin{align*}
\mu_0(\{q\in B_{2r}:\tau (q)\in B_{r}\})<\Psi(B_r).
\end{align*}
Lemma \ref{l4.5} shows that $\|\tau(0)\|\le r$.

To conclude the proof it now suffices to bound the probabilities of
the complements $\Omega\setminus C_r$ and $\Omega\setminus D_r$.
Observing that $\Psi(B_r)$ is a $1$-Lipschitz functional, we can use
Theorem \ref{t:dpp_conc} to derive that
\begin{align*}
  \BP(\Psi(B_r) \leq (\alpha-\varepsilon)\lambda_d(B_r))
\le 5\exp \left(-\frac{\varepsilon^2 \lambda_d(B_r)}{4(\varepsilon + 2\alpha)}\right).
\end{align*}
Since by Lemma~\ref{lcounting} $\mu_0(B_r)/\lambda_d(B_r)\to 1$ as $r\to\infty$ this gives the desired
inequality for $\BP(\Omega\setminus C_r)$. To treat
the complement of $D_r$ we apply Lemma \ref{l4.7}. This gives
\begin{align*}
\BE Z_r\le \varepsilon (2m)^{d}\mu_0(I_r)
\le \varepsilon \mu_0(B_{2r}).
\end{align*}
By Lemma \ref{lcounting} and the properties of intrinsic volumes
\begin{align*}
\mu_0(B_{2r})\le 2^d\lambda_d(B_r)+\sum^{d-1}_{j=0}c_j2^jr^jV_j(B_1),
\end{align*}
so that
\begin{align*}
\BE Z_r\le 3\varepsilon  2^{d-1}\lambda_d(B_{r})
\end{align*}
for all sufficiently large $r$. 
Note that $Z_r$ is a function of $\Psi \cap A$ where
$A = \cup_{q \in I_r} Q_q$. Since the $Y_q$
are bounded by $(2m)^d$ and the $Q_q$ are disjoint, we have that $Z_r$ is
$(2m)^d$-Lipschitz. Also note that $\lambda_d(A)=(2m)^d\mu_0(I_r)$. So we
can again apply the concentration inequality in Theorem \ref{t:dpp_conc} to derive that
\begin{align*}
\BP(\Omega\setminus D_r) &\le\BP(Z_r\ge \BE Z_r+ \varepsilon 2^{d-1}\lambda_d(B_{r})) \\
& \le 5\exp \left(- \frac{(\varepsilon 2^{d-1}\lambda_d(B_r))^2}{4(2m)^d(\varepsilon2^{d-1}\lambda_d(B_r)
 + 2(2m)^{2d}\alpha\mu_0(I_r))} \right) \\
& \le 5\exp 
\left(- \frac{(\varepsilon \lambda_d(B_r))^2}{8m^d(\varepsilon\lambda_d(B_r) + 2^{d+2}m^{2d}\alpha\mu_0(I_r))} \right)
\end{align*}
which in view of Lemma~\ref{lcounting} gives the desired result.\qed

\bigskip

We now return to a general stationary point process $\Psi$
with intensity $\alpha\ge 1$.
\begin{theorem}\label{t5.19}
Assume for each $n\in\N$ that the $n$-th correlation function
of $\Psi$ exists and is bounded by $n^{\theta n}c^n$ for some $c>0$ and $\theta \in [0,1)$.
Assume that there exists $c_1>0$ such that
$\BP(\|\tau(\Z^d,\Psi,0)\|>r)\le c_1 e^{-c^{-1}_1r^\delta}$
for all $r\ge 0$ and some $\delta > 0$.
Then there exists $c_2>0$ such that
\begin{align}\label{e5.3}
\BP(F(\Z^d,\Psi,0)\not\subset B_r)\le c_2e^{-c_2^{-1}r^{\beta}},\quad r>0,
\end{align}
where $\beta:=\min \{(1-\theta)\delta , 2d \}/(2d+1-\theta)$.
\end{theorem}
{\em Proof:} The proof follows closely that of Lemma~2.4
in \cite{LastPenrose2013}. Let $r\ge 1$.
We take $\varepsilon\in[0,1)$ and
a large $K\ge 4$ both to be fixed later. Let 
$m$ be the largest integer smaller than $K r^{1-\varepsilon}$
and let $E$ be the event 
that there exists a descending chain $(0,z_1,\ldots,z_m)$ in
$(\Z^d,\Psi)$ such that $\|z_1\|\le K^{-1}r^\varepsilon$.
We assert that
\begin{align}\label{e5.4}
\{F(\Z^d,\Psi,0)\not\subset B_r\}\subset \{KY>r^\varepsilon\}\cup E,
\end{align}
where $Y:=\|\tau(\Z^d,\Psi,0)\|$. 
To see this we assume that $KY\le r^\varepsilon$ and that $E$ does
not hold. Take a descending chain $(0,z_1,\ldots,z_k)$ in
$(\Z^d,\Psi)$ such that $\|z_1\|\le Y$. Since $Y\le K^{-1}r^\varepsilon$
we have $k\le m-1$ and hence
$\|z_k\|\le (m-1)Y\le (K r^{1-\varepsilon}-1)Y$. Therefore, by the definition of the matching flower, each point 
in $F(\Z^d,\Psi,0)$ has norm at most $K r^{1-\varepsilon}Y\le r$, proving
\eqref{e5.4}.

Next we bound $\BP(E)$. There is a unique
$n\in\N$ such that $m\in\{2n,2n+1\}$. By Lemma \ref{descending_chains}
we have that $\BP(E)\le a_n^n/n!$, 
where 
$$
a_n=n^{\theta}c\kappa_d^2(K^{-2}r^{2\varepsilon}+2K^{-1}r^{\varepsilon}\sqrt{d})^d.
$$
Since $n!\ge n^n e^{-n}$ (by the exponential series), we obtain
\begin{align*}
\BP(E)&\le \big(n^{\theta-1}ce\kappa_d^2(K^{-2}r^{2\varepsilon}+2K^{-1}r^{\varepsilon}\sqrt{d})^d\big)^n\\
&\le \big(n^{\theta-1}ce\kappa_d^2K^{-d}(1+2\sqrt{d})^dr^{2\varepsilon d}\big)^n,
\end{align*}
where the second inequality is due to $K\ge 1$ and $r\ge 1$.
Note that
$$
n\ge (m-1)/2\ge  (K r^{1-\varepsilon}-2)/2\ge K r^{1-\varepsilon}/4, 
$$
where the third inequality comes from $K\ge 4$ and $r\ge 1$.
Therefore
\begin{align*}
\BP(E) 
\le \big(4^{1-\theta}ce\kappa_d^2K^{\theta-1}K^{-d}r^{(1-\varepsilon)(\theta-1)}
(1+2\sqrt{d})^dr^{2\varepsilon d}\big)^n.
\end{align*}
We now set $\varepsilon:=(1-\theta)/(2d+1-\theta)$, so that
$(1-\varepsilon)(\theta-1)+2\varepsilon d=0$.
Further we choose $K$ so large that
$$
4^{1-\theta}ce\kappa_d^2(1+2\sqrt{d})^dK^{\theta-1}K^{-d}\le e^{-1}.
$$
Then
\begin{align}\label{e5.33}
\BP(E)\le \exp[-n]\le \exp[-K r^{1-\varepsilon}/4]
=\exp\big[-(K/4)r^{2d/(2d+1-\theta)}\big].
\end{align}
On the other hand we obtain from our assumptions that
\begin{align*}
\BP(KY>r^\varepsilon)
\le c_1\exp\big[-c^{-1}_1K^{-\delta}r^{\varepsilon \delta}\big]
= c_1\exp\big[-c^{-1}_1K^{-\delta}r^{(1-\theta) \delta/(2d+1-\theta)} \big].
\end{align*}
Using this and inequality \eqref{e5.33} in \eqref{e5.4}
yields the asserted inequality.
\qed

\bigskip

Next we consider the distance to the matching partner of an extra point added to 
$\Psi$. For $x \in \R^d$, recall that $\Psi^x = \Psi \cup \{x\}.$
\begin{theorem}\label{t4.13} Assume that there exists $c_1>0$ such that
$\BP(\|\tau(\Z^d,\Psi,0)\|>r)\le c_1 e^{-c^{-1}_1r^\delta}$
for all $r\ge 0$ and some $\delta > 0$. Then there exists $c_3>0$ such that 
\begin{align}\label{e5.79}
\BP(r< \|\tau(\Z^d,\Psi^x,x)-x\|<\infty)\le c_3e^{-c^{-1}_3r^\delta},\quad r>0,\,x\in\R^d.
\end{align}
\end{theorem}
{\em Proof:} Let $x\in\R^d$ and $r>0$. Suppose that for all $q\in \Z^d\setminus B(x,r)$ we have that
$\tau(\Z^d,\Psi,q) \neq x$.  Then the matching partner
of $x$ (if any) in the matching of $\Psi^x$ with $\Z^d$ must be in $B(x,r)$. Therefore either
$\tau(\Z^d,\Psi^x,x)=\infty$ or $\|\tau(\Z^d,\Psi^x,x)-x\|\le r$.
It follows that
\begin{align*}
\BP(r<\|\tau(\Z^d,\Psi^x,x)-x\|<\infty)
& \le \BP(\text{there exists $q\in \Z^d\setminus B(x,r)$ with
$\tau(\Z^d,\Psi,q) = x$}) \\
& \le \sum_{q\in \Z^d\setminus B(x,r)} \BP(\tau(\Z^d,\Psi,q) = x).
\end{align*}
Therefore we obtain from
stationarity and our assumption that
\begin{align*}
  \BP(r<\|\tau(\Z^d,\Psi^x,x)-x\|<\infty) 
& \leq \sum_{q\in \Z^d\setminus B(x,r)} \BP(\|\tau(\Z^d,\Psi,q) - q\| > \|q-x\|/2) \\ 
& = \sum_{q\in \Z^d\setminus B(x,r)} \BP(\|\tau(\Z^d,\Psi,0)\| > \|q-x\|/2) \\
& \leq \sum_{q\in \Z^d\setminus B(x,r)} c_1 e^{-c'_1 \|q-x\|^{\delta}}
\end{align*}
for some $c_1'>0$. 
In view of the monotonicity of the exponential function
there exists $b_1'>0$ such that the above right-hand side can be bounded by
\begin{align*}
b_1'\int &\I\{\|y-x\|>r\} e^{-c'_1 \|y-x\|^{\delta}}\,dy
=b_1'\int \I\{\|y\|>r\} e^{-c'_1 \|y\|^{\delta}}\,dy\\
&=d\kappa_d b_1'\int^\infty_{r}u^{d-1}e^{-c_1'u^\delta}\,du\\
&=d\kappa_d b_1'\delta^{-1}\int^\infty_{r^{\delta}}v^{(d-\delta)/\delta} e^{-c_1'v}\,dv,
\end{align*}
where the second
identity is achieved by using polar coordinates.
This yields the assertion.\qed

\bigskip

\begin{remark}\label{rem:optimality}
  Apart from the constants $c_1$ and $c_3$, the tail bounds
  in Theorem~\ref{t4.1} and Theorem~\ref{t4.13} are optimal.
  But we do not expect the tail bounds for the
  matching flower in Theorem~\ref{t5.19} to be optimal.
\end{remark}

For the case of determinantal point processes, using Corollary
\ref{c3.6} and Theorem \ref{t4.1} in Theorems \ref{t5.19} and
\ref{t4.13}, we get the following corollary.
\begin{corollary}\label{c5.21} Assume that $\Psi$ 
is a determinantal point process (in particular Poisson) with intensity $\alpha>1$.
Then there exists $c_2 > 0, c_3>0$ such that \eqref{e5.3} and \eqref{e5.79}
hold with $\beta = d/(2d+1)$ and $\delta = d$ respectively.
\end{corollary}

\subsection{Matched points in a Poisson point process}\label{s:tail_poisson}

In this subsection, we assume that $\Psi$ is a stationary Poisson point
process.
Then we can derive further tail bounds as the distribution  of $\Psi^x$
is that of the Palm distribution of the Poisson point process.
Recall the definition $\Phi := \Z^d + U$ of the randomized lattice,
where $U$ is $[0,1)^d$-valued random variable.

\begin{theorem}\label{t4.7} Assume that $U$ is uniformly distributed. Then
there exists $c_4>0$ such that 
\begin{align*}
\BP(r< \|\tau(\Phi,\Psi^x,x)-x\|<\infty)\le c_4e^{-c^{-1}_4 r^{d}},\quad r>0,\,x\in\R^d.
\end{align*}
\end{theorem}
{\em Proof:} By stationarity and translation covariance of $\tau$, it suffices
to prove the result for $x=0$. Let $r>0$. The result follows
from Theorem \ref{t4.1}, once we have shown that
\begin{align*}
\alpha\,\BP(r<\|\tau(\Phi,\Psi^0,0)\|<\infty)=\BP(r<\|\tau(\Z^d,\Psi,0)\|).
\end{align*}
Even though this can be derived from general results in \cite{LaTho09},
we prefer to give a direct argument. Let
$g\colon \R^d\to[0,\infty)$ be a measurable function
and let $W:=[0,1)^d$. By stationarity, independence of $\Psi$ and
$\Phi$ and the Mecke equation (Theorem 4.1 in \cite{LastPenrose17})
we have
\begin{align*}
  \alpha\,\BE&\I\{\tau(\Phi,\Psi^0,0)\ne \infty\}g(\tau(\Phi,\Psi^0,0))\\
&=\alpha\,\BE\int_W\I\{\tau(\Phi-x,(\Psi-x)\cup\{0\},0)\ne \infty\}g(\tau(\Phi-x,(\Psi-x)\cup\{0\},0))\,dx\\
&=\alpha\,\BE\int_W\I\{\tau(\Phi,\Psi^x,x)\ne \infty\}g(\tau(\Phi,\Psi^x,x)-x)\,dx\\
&=\BE\sum_{x\in\Psi\cap W}\I\{\tau(\Phi,\Psi,x)\ne \infty\}g(\tau(\Phi,\Psi,x)-x)
=\BE\sum_{x\in \Psi^{\tau} \cap W}g(\tau(\Phi,\Psi,x)-x).
\end{align*}
Since all points of $\Phi$ are almost surely matched, the above equals
\begin{align*}
\BE&\sum_{p\in\Phi}\I\{\tau(\Phi,\Psi,p)\in W\}g(p-\tau(\Phi,\Psi,p))\\
&=\BE\sum_{p\in\Phi}\I\{\tau(\Phi-p,\Psi-p,0)+p\in W\}g(-\tau(\Phi-p,\Psi-p,0))\\
&=\BE\sum_{p\in\Phi}\I\{\tau(\Z^d,\Psi,0)+p\in W\}g(-\tau(\Z^d,\Psi,0)),
\end{align*}
where we have used independence of $\Phi$ and $\Psi$ and stationarity of
$\Psi$. By Campbell's formula for $\Phi$ (Proposition 2.7 in \cite{LastPenrose17}) 
this comes to
\begin{align*}
\BE\int\I\{\tau(\Z^d,\Psi,0)+x\in W\}g(-\tau(\Z^d,\Psi,0))\,dx.
\end{align*}
{Since $W$ has unit volume,} we finally obtain that
\begin{align}
\alpha\,\BE\I\{\tau(\Phi,\Psi^0,0)\ne \infty)g(\tau(\Phi,\Psi^0,0))
=\BE g(-\tau(\Z^d,\Psi,0))
\end{align}
and in particular the desired result.\qed

\begin{proposition}\label{p4.5} Assume that $U$ is uniformly distributed.
Then there exists $c_5>0$ such that 
\begin{align*}
\BP(M(\Phi,\Psi^x,x)=1,F(\Phi,\Psi^x,x)\not\subset B(x,r))\le c_5 e^{-c^{-1}_5r^{\beta'}},
\quad r>0,\,x\in\R^d,
\end{align*}
where $\beta' := d/(2d+1)$ and the matching status $M$ was defined in \eqref{e3.7}.
\end{proposition}
{\em Proof:} Given Theorem \ref{t4.7},
the proof is basically the same as that of Theorem \ref{t5.19} with $\theta = 0, \delta = d$.
We define $Y:=\|\tau(\Z^d,\Psi^x,x)-x\|$ and replace
the set inclusion \eqref{e5.4} by
\begin{align*}
\{M(\Phi,\Psi^x,x)=1,F(\Z^d,\Psi^x,x)\not\subset B(x,r)\}
\subset \{M(\Phi,\Psi^x,x)=1,KY>r^\varepsilon\}\cup E.
\end{align*}
To bound $\BP(E)$ we apply Lemma \ref{descending_chains2} instead
of Lemma \ref{descending_chains}.
\qed

\section{Hyperuniformity}\label{sec:hyperuniformity}

Again we assume that $\Psi$ is a stationary point process with
intensity $\alpha \geq 1$ and that $\Phi=\Z^d$. Recall the definition of hyperuniformity in
\eqref{ehyperun}, the definition of the $\alpha_{p,q}$-mixing
coefficients in Section \ref{s:app} {and that $\Psi^{\tau}_0$ denotes the set of matched points in $\Psi$ when $\Z^d$ is matched with $\Psi$.}
We abbreviate $F(q):=F(\Z^d,\Psi,q)$, $q\in\Z^d$.

\begin{theorem}\label{t:hypuniform_mixing}
  Let $\Psi$ be a stationary point process with intensity
  $\alpha \geq 1$. Assume that for some $\gamma \in (0,1]$ we have
  that
\[ 
\sum_{w \in \Z^d}\|w\|^{\gamma
  d}\alpha_{\|w\|^{\gamma d},\|w\|^{\gamma d}}(\|w\|) < \infty \, \,
\mbox{and} \, \, \sum_{w \in \Z^d} \BP(F(0)\not\subset B(0,\|w\|^{\gamma})) < \infty.
\]
Then the point process $\Psi^{\tau}_0$ is hyperuniform.
\end{theorem}
Before presenting the proof, we give two simple corollaries. If $\Psi$ has unit intensity, then using Theorem
\ref{2.11}, we have that $\Psi^{\tau}_0 = \Psi$ and thus we have the following sufficient condition for hyperuniformity.
\begin{corollary}
\label{c:hyp_suff}
Let $\Psi$ be a unit intensity stationary point process satisfying the
assumptions in Theorem \ref{t:hypuniform_mixing}. Then $\Psi$ is
hyperuniform.
\end{corollary}
\begin{theorem}\label{t:dpp_hyp}
  Let $\Psi$ be a determinantal point process (in particular Poisson) with intensity
  $\alpha > 1$ with kernel $K$ and define
$\omega(s) := \sup_{x,y \in \R^d, \|x-y\| \geq s} |K(x,y)|$, $s >0$.
Assume that there exists $\gamma \in (0,1]$ such that
  $\sum_{w \in \Z^d}\|w\|^{\gamma}\omega(\|w\|)^2 < \infty$. 
   Then $\Psi^{\tau}_0$ is hyperuniform.
\end{theorem}
\begin{proof}
{The proof follows from combining Corollary \ref{c5.21} and Theorems
\ref{t:hypuniform_mixing} and \ref{t:dpp_mixing}.}
\end{proof}
Many determinantal point processes satisfy the above decay condition
on the kernel including Poisson (an extreme case of determinantal
point process) and the Ginibre point process (see \cite{Hough09}).

\begin{proof}[Proof of Theorem \ref{t:hypuniform_mixing}]
As before, we write $\tau(x):=\tau(\Z^d,\Psi,x)$, $x\in\R^d$ and $Y_x:=\tau(x)-x$.

For each Borel set $B\subset\R^d$ we have that
\begin{align*}
\BE \Psi^{\tau}_0(B)=\BE\sum_{p\in\Z^d}\I\{\tau(p)\in B\}=\sum_{p\in\Z^d}\BP(p+Y_p\in B).
\end{align*}
Furthermore,
\begin{align*}
\BE \Psi^{\tau}_0(B)^2&=\BE \Psi^{\tau}_0(B)+\BE\sideset{}{^{\ne}}\sum_{p,q\in\Z^d}\I\{\tau(p)\in B,\tau(q)\in B\}\\
&=\BE \Psi^{\tau}_0(B)+\sideset{}{^{\ne}}\sum_{p,q\in\Z^d}\BP(p+Y_p\in B,q+Y_q\in B).
\end{align*}
This shows that
\begin{align*}
\BV[&\Psi^{\tau}_0(B)]=\sum_{p\in\Z^d}\BP(p+Y_p\in B)
-\sum_{p\in\Z^d}\BP(p+Y_p\in B)^2\\
&+\sideset{}{^{\ne}}\sum_{p,q\in\Z^d}(\BP(p+Y_p\in B,q+Y_q\in B)-\BP(p+Y_p\in B)\BP(q+Y_q\in B)).
\end{align*}
The series convergence will be justified below.

Take a convex compact set $W\subset\R^d$ containing the origin $0$ in its
interior. Let $r>0$ and set $W_r:=rW$. By stationarity,
\begin{align*}
\lambda_d(W_r)^{-1}&\sum_{p\in\Z^d}\BP(p+Y_p\in W_r)
=\lambda_d(W_r)^{-1}\sum_{p\in\Z^d}\BP(p+Y_0\in W_r)\\
&=\lambda_d(W_r)^{-1}\BE \sum_{p\in\Z^d}\I\{p+Y_0\in W_r\}\\
&=\int\lambda_d(W_r)^{-1}\mu_0(W_r-y)\,\BP(Y_0\in dy).
\end{align*}
As $r\to\infty$ the above integrand tends to $1$ uniformly in $y$ by Lemma~\ref{lcounting}, so that
\begin{align*}
\lim_{r\to\infty}\lambda_d(W_r)^{-1}\sum_{p\in\Z^d}\BP(p+Y_p\in W_r)=1.
\end{align*}
Similarly,
\begin{align}
\lambda_d(W_r)^{-1}&\sum_{p\in\Z^d}\BP(p+Y_p\in W_r)^2\nonumber\\
&=\lambda_d(W_r)^{-1}\iint \sum_{p\in\Z^d}\I\{p+y\in W_r,p+z\in W_r\}
\,\BP(Y_0\in dy)\,\BP(Y_0\in dz)\nonumber\\
&=\iint \lambda_d(W_r)^{-1}\mu_0((W_r-y)\cap (W_r-z))
\,\BP(Y_0\in dy)\,\BP(Y_0\in dz).
\label{tends-to-1}
\end{align}
Since by Lemma~\ref{lcounting} and standard properties of the intrinsic volumes
\begin{align*}
&\frac{|\mu_0((W_r-y)\cap (W_r-z))-\lambda_d((W_r-y)\cap (W_r-z))|}{\lambda_d(W_r)}\\
&\qquad\qquad\qquad\le\sum_{i=0}^{d-1}c_i\frac{V_i((W_r-y)\cap (W_r-z))}{\lambda_d(W_r)}
\le\sum_{i=0}^{d-1}c_i\frac{V_i(W_r)}{\lambda_d(W_r)}=\sum_{i=0}^{d-1}c_i\frac{V_i(W)}{\lambda_d(W)} r^{i-d},
\end{align*}
the right hand side of \eqref{tends-to-1} tends to $1$ as $r\to\infty$.
Therefore we have
\begin{align*}
\sigma^2:=&\lim_{r\to\infty}\lambda_d(W_r)^{-1} \BV \Psi^{\tau}_0(W_r)\\
=&\lim_{r\to\infty}\lambda_d(W_r)^{-1}
\sideset{}{^{\ne}}\sum_{p,q\in\Z^d}(\BP(p+Y_0\in W_r,q+Y_{q-p}\in W_r)\\
&\hspace{3.8cm}-\BP(p+Y_0\in W_r)\BP(q+Y_0\in W_r)),
\end{align*}
where we have again used stationarity and where the existence of the limit
will be justified below. This can be rewritten as
\begin{align}\label{e7.1}
\sigma^2=\lim_{r\to\infty}\lambda_d(W_r)^{-1}\sum_{w\in\Z^d\setminus\{0\}}
\sum_{p\in\Z^d} a(r,p,w),
\end{align}
where
\begin{align*}
a(r,p,w):=\BP(p+Y_0\in W_r,p+w+Y_{w}\in W_r)-\BP(p+Y_0\in W_r)\BP(p+w+Y_0\in W_r).
\end{align*}

Let $w\in\Z^d\setminus\{0\}$. Then
\begin{align*}
\sum_{p\in\Z^d}&\BP(p+Y_0\in W_r,p+w+Y_w\in W_r)\\
&=\int \sum_{p\in\Z^d}\I\{p+y\in W_r,p+w+z\in W_r\}\, \BP((Y_0,Y_w)\in d(y,z))\\
&=\int \mu_0((W_r-y)\cap (W_r-w-z))\, \BP((Y_0,Y_w)\in d(y,z)).
\end{align*}
Therefore, we obtain similarly as above
\begin{align*}
\lim_{r\to\infty}\lambda_d(W_r)^{-1}\sum_{p\in\Z^d}&\BP(p+Y_p\in W_r,p+w+Y_{p+w}\in W_r)=1.
\end{align*}
Further
\begin{align*}
\sum_{p\in\Z^d}\BP(p+Y_0\in W_r)&\BP(p+w+Y_0\in W_r)\\
&=\iint \mu_0((W_r-y)\cap (W_r-w-z))\,\BP(Y_0\in dy)\,\BP(Y_0\in dz),
\end{align*}
so that
\begin{align*}
\lim_{r\to\infty}\lambda_d(W_r)^{-1}\sum_{p\in\Z^d}\BP(p+Y_0\in W_r)\BP(p+w+Y_0\in W_r)=1.
\end{align*}

We next show that
\begin{align}
\lambda_d(W_r)^{-1}\sum_{p\in\Z^d} |a(r,p,w)|
\le b(w),\quad w\in\Z^d\setminus\{0\},\,r>0,
\end{align}
where the numbers $b(w)\ge 0$, $w\in\Z^d\setminus\{0\}$, 
do not depend on $r$ and satisfy
$\sum_w b(w)<\infty$. Then we can apply dominated convergence to
\eqref{e7.1} to conclude that $\sigma^2=0$.

Let $\gamma \in (0,1]$ be as assumed in the theorem. By the stopping set
properties of the flowers and
from the definition of the $\alpha_{p,q}$-mixing coefficients
in Section \ref{s:app}, we have that   
\begin{align} | & \BP(F(0)\subset B(0,\|w\|^{\gamma}/3),p+Y_0\in W_r,F(w)\subset B(w,\|w\|^{\gamma}/3),p+w+Y_{w}\in W_r)  \notag \\
& - \BP(F(0)\subset B(0,\|w\|^{\gamma}/3),p+Y_0\in W_r)\BP(F(w)\subset B(w,\|w\|^{\gamma}/3),p+w+Y_{w}\in W_r)|  \notag \\
& \leq \alpha_{\|w\|^{\gamma d},\|w\|^{\gamma d}}(\|w\|/3), \label{e:flower_bd}
\end{align}
and again from the stopping set property of the matching flower, we obtain that
\begin{align*}
\{F(0)\subset B(0,\|w\|^{\gamma}/3),p+Y_0\in W_r \} &\subset \{ Y_0 \in B(0,\|w\|^{\gamma}/3),p+Y_0\in W_r \}\\
                                                    &\subset \{ p \in W_r + B(0,\|w\|^{\gamma}/3) \}.
\end{align*}
Defining
\begin{align*}
&\tilde{a}(r,p,w)\\
&:=\BP(F(0)\subset B(0,\|w\|^{\gamma}/3),p+Y_0\in W_r,
F(w)\subset B(w,\|w\|^{\gamma}/3),p+w+Y_{w}\in W_r) \\
&\quad\,-\BP(F(0)\subset B(0,\|w\|^{\gamma}/3),p+Y_0\in W_r)
\BP(F(w)\subset B(w,\|w\|^{\gamma}/3),p+w+Y_{w}\in W_r)
\end{align*}
we have that $\tilde{a}(r,p,w) = 0$ for $p \notin  W_r + B(0,\|w\|^{\gamma}/3)$ as both the terms are $0$. Thus, we can derive from \eqref{e:flower_bd} that
\[  \sum_p |\tilde{a}(r,p,w)| \leq \alpha_{\|w\|^{\gamma d},\|w\|^{\gamma d}}(\|w\|/3)\lambda_d(W_r + B(0,\|w\|^{\gamma}/3)).\]
Hence, we have that
\begin{equation}
\label{e:ta_term}
\lambda_d(W_r)^{-1}\sum_p |\tilde{a}(r,p,w)| \leq C' \|w\|^{\gamma d} \alpha_{\|w\|^{\gamma d},\|w\|^{\gamma d}}(\|w\|/3).
\end{equation}
Now, we bound the remaining terms in $\sum_{p\in\Z^d}|a(r,p,w)|$. We have (up to boundary effects)
\begin{align*}
\lambda_d(W_r)^{-1}&\sum_{p\in\Z^d}\BP(F(0)\not\subset B(0,\|w\|^{\gamma}/3),p+Y_0\in W_r,p+w+Y_{w}\in W_r)\\
&\le \BE\I\{F(0)\not\subset B(0,\|w\|^{\gamma}/3)\}\lambda_d(W_r)^{-1}\sum_{p\in\Z^d}\I\{p+Y_0\in W_r\} \\
&\le \BP(F(0)\not\subset B(0,\|w\|^{\gamma}/3)) =:c(w).
\end{align*}
Similarly we obtain that
\begin{align*}
\lambda_d(W_r)^{-1}&\sum_{p\in\Z^d}\BP(F(w)\not\subset B(w,\|w\|^{\gamma}/3),p+Y_0\in W_r,p+w+Y_{w}\in W_r)
\le c(w).
\end{align*}
By the additivity of a probability measure and \eqref{e:ta_term}, we now obtain
\begin{align*}
\lambda_d(W_r)^{-1}\sum_{p\in\Z^d}|a(r,p,w)|\le 4 c(w) + C' \|w\|^{\gamma d} \alpha_{\|w\|^{\gamma d},\|w\|^{\gamma d}}(\|w\|/3),\quad w\in\Z^d\setminus\{0\},
\end{align*}
and the proof is complete as both terms on the above right-hand side are
summable by our assumptions.
\end{proof}

\section{Number rigidity}\label{srigid}
We assume here again $\Phi = \Z^d$. Recall that a point process $\Psi'$ on
$\R^d$ is called {\em number  rigid} if for each closed ball $A\subset\R^d$
centred at the origin,
the random number $\Psi'(A)$ of $\Psi'$-points in $A$ is almost surely
determined by $\Psi'\cap A^c$. In this case the latter property holds
for all bounded Borel sets $B\subset\R^d$. We now give a condition for
number rigidity of the point processes of matched points. 
Recall the definition $F(q):=F(\Z^d,\Psi,q)$, $q\in\Z^d$.

\begin{theorem}\label{t8.1}
Assume that $\Psi$ is a stationary point process with intensity $\alpha
\geq 1$ such that $\sum_{w \in \Z^d} \BP(F(0)\not\subset B(0,\|w\|)) <
\infty$. Then $\Psi^\tau_0$ is a number rigid point process.  
\end{theorem}
{Before moving onto the proof, we state an important consequence of the theorem.
\begin{theorem}
\label{t:num_rigid_dpp}
Assume that $\Psi$ is a stationary determinantal point process (in particular
Poisson) with intensity $\alpha > 1$. Then the matched
point process $\Psi^{\tau}_0$ is number rigid.
\end{theorem}
\begin{proof}
{It follows from combining Corollary \ref{c5.21} and Theorem \ref{t8.1}.}
\end{proof}
}
\begin{proof}[Proof of Theorem \ref{t8.1}]
Write $\Psi_0:=\Psi^\tau_0$.
Let $A\subset\R^d$ be a closed ball with radius $r$ centred at
the origin.
We wish to show that $\Psi_0(A)$ is almost surely determined
by $\Psi_0\cap A^c$. For this we use a modified version
of the mutual nearest neighbour matching of 
$\Z^d$ with $\Psi_0\cap A^c$, using the concept of {\em undecided}
points. We call the set of undecided points of $\Psi_0$ as $\Psi_0^{UD}$
and that of $\Phi$ as $\Phi^{UD}$.  The following points are undecided :
(i) all points in $\Z^d\cap A$, (ii) all points in $\Z^d\cap A^c$ or in
$\Psi_0\cap A^c$ who need to consider points inside $A$ for deciding
their potential matching partner. More precisely, we can write 
this as
\begin{align*}
\Psi_0^{UD} &= \{ x \in \Psi_0 \cap A^c :  F(\Z^d,\Psi_0\cap A^c,x) \cap  A \neq \emptyset \},  \\
\Phi^{UD} & = ( \Z^d \cap A ) \cup \{q \in \Z^d \cap A^c : F(\Z^d,\Psi_0\cap A^c,q)\cap A\neq\emptyset \}.
\end{align*}
We have illustrated the notion of undecided points in
Fig.~\ref{fig:undecided}. For $z \in (\Z^d \cup \Psi_0) \cap A^c$,
by the stopping set property of the matching flower
(Lemma~\ref{matching-flower-is-stopping-set}), we have that
$F(\Z^d,\Psi_0,z) \subset A^c$ iff
$F(\Z^d \cap A^c,\Psi_0 \cap A^c,z) \subset A^c$. Thus, we have that
$\Phi^{UD},\Psi_0^{UD}$ are determined by $(\Z^d,\Psi_0 \cap
A^c)$.
Further, by Proposition \ref{matching-of-psi-tau} the points in $\Z^d \setminus \Phi^{UD}$ are
matched to points in $\Psi_0 \setminus \Psi_0^{UD}$ and vice-versa.
\begin{figure}[t]
\centering
\includegraphics[width=\textwidth]{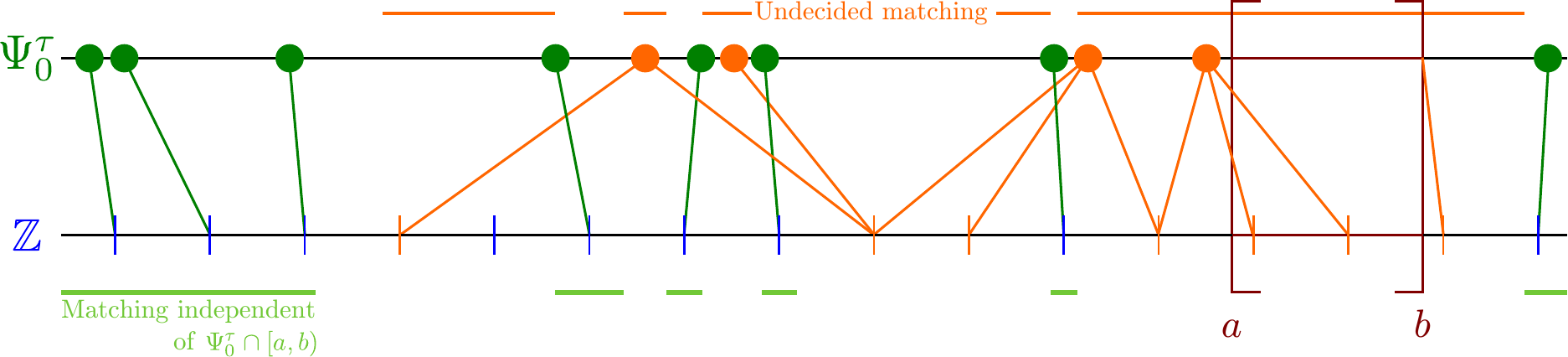}
\caption{Undecided matching: $\Psi^\tau_0$ is only observed outside of $A=[a,b]$.
	There are points (blue and green) for which the iterated mutually closest matching is independent of $\Psi^\tau_0\cap A$.
	The potential matching partners of the other points (marked in orange) are undecided.}
\label{fig:undecided}
\end{figure}
	
Let us now define $Z_1 := \card (\Phi^{UD} \cap A^c)$, 
$Z_2 : = \card(\Phi^{UD} \cap A)$, $Z_3 := \card \Psi_0^{UD}$.   Since $\Z^d$ is
completely matched to $\Psi_0$, we have that $Z_3 \leq Z_1 + Z_2$. Now
if we show that $Z_1 < \infty$ a.s., then we have by the above
inequality that $Z_3 < \infty$ a.s. and again as a consequence of the
complete matching of $\Z^d$ to $\Psi_0$, we derive that $\Psi_0(A) = Z_1
+ Z_2 - Z_3$ a.s. Thus, we have that $\Psi_0(A)$ is a.s.\ determined by
$(\Z^d,\Psi_0 \cap A^c)${,} i.e., $\Psi_0$ is a number rigid point process.
So our proof is now complete if we show that $Z_1 < \infty$ a.s..

By the stopping set property of the matching flower
(Lemma~\ref{matching-flower-is-stopping-set}) as mentioned above, we
have that
\[
  \Phi^{UD} \cap A^c = \{q\in\Z^d\cap A^c:F(\Z^d,\Psi_0,q)\cap A\ne\emptyset\}.
\]
Now,
\begin{align*}
\BE Z_1 = \BE \card (\Phi^{UD} \cap A^c) &=  \sum_{q\in\Z^d,\|q\|>r}\BP(F(\Z^d,\Psi_0,q)\cap A\ne\emptyset) \\
&=\sum_{q\in\Z^d,\|q\|>r}\BP((F(\Z^d,\Psi_0,0)+q)\cap A\ne\emptyset)\\
&=\sum_{q\in\Z^d,\|q\|>r}\BP(F(\Z^d,\Psi_0,0)\cap B(-q,r)\ne\emptyset)\\
&\le \sum_{q\in\Z^d,\|q\|>r}\BP(F(\Z^d,\Psi_0,0)\not\subset B(0,\|q\|-r)),
\end{align*}
where the last inequality follows from the fact that the balls $B(-q,r)$
and $B(0,\|q\|-r)$ have disjoint interiors.
Proposition \ref{matching-of-psi-tau}
shows that $\tau(\Z^d,\Psi_0,0)=\tau(\Z^d,\Psi,0)$.
Therefore any competing chain in $(\Z^d,\Psi_0)$ for $0$ is also
a competing chain in $(\Z^d,\Psi)$ for $0$ so that
we have $F(Z^d,\Psi_0,0)\subset F(Z^d,\Psi,0)$.
Hence we obtain that
\begin{align*}
\BE  Z_1\le \sum_{q\in\Z^d,\|q\|>r}\BP(F(\Z^d,\Psi,0)\not\subset B(0,\|q\|-r)).
\end{align*}
Now by our assumption the
last sum is finite and this yields that $Z_1 < \infty$ a.s. This
completes the proof.
\end{proof}

\section{Pair correlation function in the Poisson case}\label{spair}

In this section, we assume that $\Psi$ is a Poisson process
with intensity $\alpha>1$. 

For $x,y\in\R^d$ and $\varphi\in\mathbf{N}$, we write $\varphi^{x,y}:=\varphi\cup\{x,y\}$.
Recall the definition of matching status in \eqref{e3.7} and the definitions of
the intensity function 
and the second order correlation function of a point process from the appendix.
\begin{proposition}\label{p4.1} The function
\begin{align*}
\rho_1(x):=\alpha\,\BE M(\Psi^{x},\Z^d,x),\quad x\in\R^d,
\end{align*}
is an intensity function of $\Psi^\tau$, while
\begin{align*}
\rho_2(x,y):=\alpha^2\,\BE M(\Psi^{x,y},\Z^d,x)M(\Psi^{x,y},\Z^d,y),\quad x,y\in\R^d,
\end{align*}
is a second order correlation function of $\Psi_0^\tau$.
\end{proposition}

{\em Proof:} The first equation is a quick consequence of
the Mecke equation (Theorem 4.1 in \cite{LastPenrose17})
while the second follows from the bivariate Mecke equation, that is the case
$m=2$ of Theorem 4.4 in \cite{LastPenrose17}.\qed

\bigskip

The next result shows that the {\em truncated second order correlation function}
decays exponentially fast.

\begin{proposition}\label{p4.7} 
There exist $a_1> 0$ such that 
\begin{align*}
|\rho_2(x,y)-\rho_1(x)\rho_1(y)|
\le a_1e^{-a^{-1}_1\|x-y\|^{\beta'}},\quad x,y\in\R^d,
\end{align*}
where $\rho_1$ and $\rho_2$ are as in Proposition \ref{p4.1}
and $\beta' = d/(2d+1)$.
\end{proposition}
{\em Proof:} 
Let $x,y\in\R^d$.
By the stopping set property of the matching flower and complete independence of $\Psi$,
the events $\{F(\Z^d,\Psi^x,x)\subset B(x,\|y-x\|/3)\}$ and
$\{F(\Psi^y,\Z^d,y)\subset B(y,\|y-x\|/3)\}$ are independent. Moreover,
on the first event we have that
$M(\Psi^{x,y},\Z^d,x)$ $=M(\Psi^{x},\Z^d,x)$. 
A similar statement applies to the second event. Therefore
it follows that
\begin{align*}
|\BE &M(\Psi^{x,y},\Z^d,x)M(\Psi^{x,y},\Z^d,y)-\BE M(\Psi^{x},\Z^d,x)\BE M(\Psi^{y},\Z^d,y)|\\
&\le 2\, \BP(M(\Psi^{x},\Z^d,x)=1,F(\Z^d,\Psi^x,x)\not\subset B(x,\|y-x\|/3))\\
&\quad + 2\, \BP(M(\Psi^{y},\Z^d,y)=1,F(\Psi^y,\Z^d,y)\not\subset B(y,\|y-x\|/3)).
\end{align*}
We claim that
\begin{align}\label{e6.2}
\BP(M(\Psi^{x},\Z^d,x)=1,F(\Z^d,\Psi^x,x)\not\subset B(x,r))
\le b'_1e^{-b'_2r^{\beta'}}, \quad r>0,\,x\in\R^d.
\end{align}
for some constants $b'_1,b'_2>0$. This would imply the assertion.

To check \eqref{e6.2}, we need to derive the equivalent of
Proposition \ref{p4.5} for non-randomized lattices. However, instead
of using Theorem \ref{t4.7}, we use Corollary \ref{c5.21} and follow
the proof of Theorem~\ref{t5.19}. Since $\theta = 0$ for the Poisson point process, $\varepsilon = 1/(2d+1)$.  But since decay rate provided by
Corollary \ref{c5.21} is same as that of Theorem \ref{t4.1}, we obtain the
exponent $\beta'=\varepsilon d$.
\qed

\begin{remark}\label{asymptotics}\rm
Let $W$ be an arbitrary convex and compact set with $\lambda_d(W)>0$ and
$0 \in W$.
From hyperuniformity ({Theorem~\ref{t:dpp_hyp}}), exponential decay of
truncated second order correlation function (Proposition \ref{p4.7}) and
using Proposition 2 of \cite{Martin1980}, we have that
$\BV\Psi_0^\tau(rW) = O(r^{d-1})$ {(that is, $\limsup_{r \to
\infty}r^{-(d-1)}\BV\Psi_0^\tau(rW) < \infty$)}
and further using Lemma 1.6 of \cite{Nazarov2012} for
$d = 2$, we have that $\BV\Psi_0^\tau(rW) = \Theta(r)$
{(that is, $\limsup_{r \to \infty}r^{-1}\BV\Psi_0^\tau(rW) < \infty$
and $\limsup_{r \to \infty}r/\BV\Psi_0^\tau(rW) < \infty$).}
These are so-called class~I hyperuniform systems;
see~\cite{Torquato2018}.
\end{remark}

\section{One-sided matching on the line}\label{secone-sided}

In this section, we shall consider the one-sided matching on the line.
The main motivation for such a consideration is to generalize our
hyperuniformity results for stable partial matchings.
Our set-up is to consider a partial matching of a $\Z$-stationary point process $\Phi$
with a $\Z$-stationary point process $\Psi$ on $\R$.
{As in general dimensions, we say that a point process $\Phi$
on $\R$ is is $\Z$-stationary if $\Phi+z$ has the same distribution as
$\Phi$ for all $z\in\Z^d$.}
The main difference with the previous sections shall be that the
matching will be one-sided{,} i.e., every $\Phi$ point will be
matched to its right.
Every $\Phi$ point is matched with the first `available' $\Psi$ point on 
its right. In other words, the $\Phi$ points search for their
`partners' to their right and the $\Psi$ points search for their
`partners' to their left.  One can describe a mutual nearest-neighbour
matching algorithm as for stable matching (see Section \ref{sdefinition}) but with $\Phi$ points searching only to their right and $\Psi$ points searching only to their left. The matched process $\Psi_0 := \Psi_0^{\tau} \subset \Psi$ can be interpreted as {\em the output process} of a queueing system with arrivals at $\Phi$ and potential departures at $\Psi$ with LIFO (Last-In-First-Out) rule. Under certain additional assumptions, we shall now show that $\Psi_0$ is hyperuniform iff $\Phi$ is.
Departure processes of queueing systems under various service rules and
i.i.d.~inter-arrival times have been studied extensively in queueing
theory and in particular, hyperuniformity of $\Psi_0$ has been proven in
the case of $\Phi = \Z_+$ (i.e., periodic arrivals) (see \cite{HMNW11,Whitt1984}).
From a more point process perspective, this point process has been
investigated in \cite{GoldsteinLebowitzSpeer2006}.  Here, we interpret
the departure process as arising from a partial matching and hence
naturally extend to general arrival and potential departure processes.

We use the notation as before for describing the matching. For $p \in
\Phi$, let $\tau(\Phi,\Psi,p) \in \Psi$ denote the  matched point of
$\Psi$ and for $x \in \Psi$, $\tau(\Phi,\Psi,x) \in \Phi$ denotes the
matched point of $\Phi$ if it exists, else we set $\tau(\Phi,\Psi,x) =
\infty$. By our assumption on one-sided matching, we have that $Y_p :=
\tau(\Phi,\Psi,p) - p \geq 0$ for all $p \in \Phi$. By stationarity of
$\Psi$ and $\Phi$, we have that the random field $\{Y_p\}_{p \in \Phi}$
and the point process $\Psi_0$ are $\Z$-stationary. For $t \in \Z$, we
define $L(t)$ as the number of un-matched points of $\Phi$ up to
$t${,} i.e.,
$$ L(t) := \sum_{p \in \Phi, p < t} \1\{ p + Y_p \geq t\} =  \sum_{p \in \Phi - t, p - t < 0} \1\{ p - t + Y_p \geq 0\} \stackrel{d}{=} \sum_{p \in \Phi, p < 0} \1\{ p + Y_p \geq 0\},$$ 
where the last equality follows from $\Z$-stationarity of $\Psi$ and
$\{Y_p\}_{p \in \Phi}$. Thus, we see that $\{L(t)\}_{t \in \Z}$ is also
a $\Z$-stationary random field. Now, we have our main theorem on
one-sided matchings.
\begin{theorem}
\label{t:hyp_1s}
Let $\Phi$ and $\Psi$ be $\Z$-ergodic point processes on $\R$
with intensities $1$ and $\alpha (\geq 1)$ respectively.
Consider the one-sided matching of $\Phi$ and $\Psi$ as
described above. Also, assume that $\BV(L(0)) < \infty$. Then, we have the following :
\begin{enumerate}
\item If $\lim_{t \in \N, t \to \infty} \BV(\Phi([-t,t))) = \infty$ then we have that
\[ \lim_{t \in \N, t \to \infty} \frac{\BV(\Psi_0([-t,t)))}{\BV(\Phi([-t,t)))} = 1 .\]
\item If $\BV(\Phi([-t,t)))$ is bounded then so is $\BV(\Psi_0([-t,t)))$.
\end{enumerate}
\end{theorem}
In other words, we have that $\Psi_0$ is hyperuniform iff $\Phi$ is
hyperuniform under appropriate assumptions. Though it is surprising that
there is no explicit requirement on tail bounds for the $Y_p$, the finite
variance condition shall impose the same.  Suppose that $\Phi = \Z$ and
then by stationarity of the $Y_p$, we have that
\[\BE(L(0)) = \sum_{p \in \Z, p < 0}\BP(Y_p \geq -p) = \sum_{p \in \Z, p < 0}\BP(Y_0 \geq -p) = E(Y_0).\]
Thus $E(Y_0) < \infty$ is a necessary condition for $\BV(L(0))$ to be
finite. After the proof, we shall also give a sufficient condition for
the finiteness of the variance when $\Phi = \Z$. 
\begin{proof} {The proof idea is similar to that of \cite[Theorem 1.1]{Whitt1984}.} We shall derive suitable bounds on $\BV(\Psi_0([-t,t)))$ from which both the statements will follow trivially. We shall assume that always $t \in \N$. From the representation of $L(t)$, we have that 
$$ \Psi_0[-t,t) = \sum_{p \in \Phi} \1\{ -t \leq p + Y_p < t\} = L(-t) - L(t) + \Phi([-t,t)), \, \, \, t \in \N$$
and so we obtain that
\[ \BV(\Psi_0([-t,t)))  = \BV(L(-t) - L(t)) + 2\COV{L(-t)-L(t),\Phi([-t,t))}  + \BV(\Phi([-t,t))) \]
Now again by $\Z$-stationarity of $L(.)$ and Cauchy-Schwarz, we derive that
\[  \BV(L(-t) - L(t)) = 2\BV(L(0)) - 2\COV{L(0),L(2t)} \leq 4\BV(L(0)), \]
and
\[ 2|\COV{L(-t)-L(t),\Phi([-t,t))}| \leq 4 \sqrt{\BV(\Phi([-t,t))\BV(L(0))}. \]
Thus, both the statements in 1.) and 2.) follow from the finiteness of $\BV(L(0))$ and our respective assumptions on $\BV(\Phi([-t,t)))$.
\end{proof}
\begin{theorem}
\label{t:var_zd}
Let $\Phi = \Z$ and $\Psi$ be a $\Z$-stationary point process on
$\R$ with intensity $\alpha \geq 1$ such that $\sum_{p >
0, p \in \Z}\sqrt{\BP(\Psi([0,p]) < p)} < \infty$. Then we have that
$\BE(L(0)^2) < \infty$. 
\end{theorem}
The condition on $\Psi$ is satisfied by stationary determinantal point
processes (including Poisson) with intensity $\alpha > 1$ (see Theorem
\ref{t:dpp_conc}). Many other stationary point processes called as {\em
weak sub-Poisson point processes} and with intensity $\alpha > 1$  also
satisfy the condition in Theorem \ref{t:var_zd} (see Proposition 2.3 in
\cite{Blaszczy15}). Determinantal point processes are examples of the
so-called weak sub-Poisson point processes. 
\begin{proof}
By the representation of $L(0)$, stationarity of the $Y_p$ and Cauchy-Schwarz inequality, we have that
\[ \BE(L(0)^2) \leq \left(\sum_{p \in \Z, p < 0} \sqrt{\BP(Y_0 \geq -p)} \right)^2. \]
The proof is now complete by noting that for $p \in \Z, p > 0$, 
$\{Y_0 > p\} = \{\Psi([0,p]) < p\}$ due to the definition of the one-sided matching.
\end{proof}

\section{Simulated matching on a torus}
\label{sec_simulations}

\begin{figure}[t]
  \centering
  \includegraphics[width=0.49\textwidth]{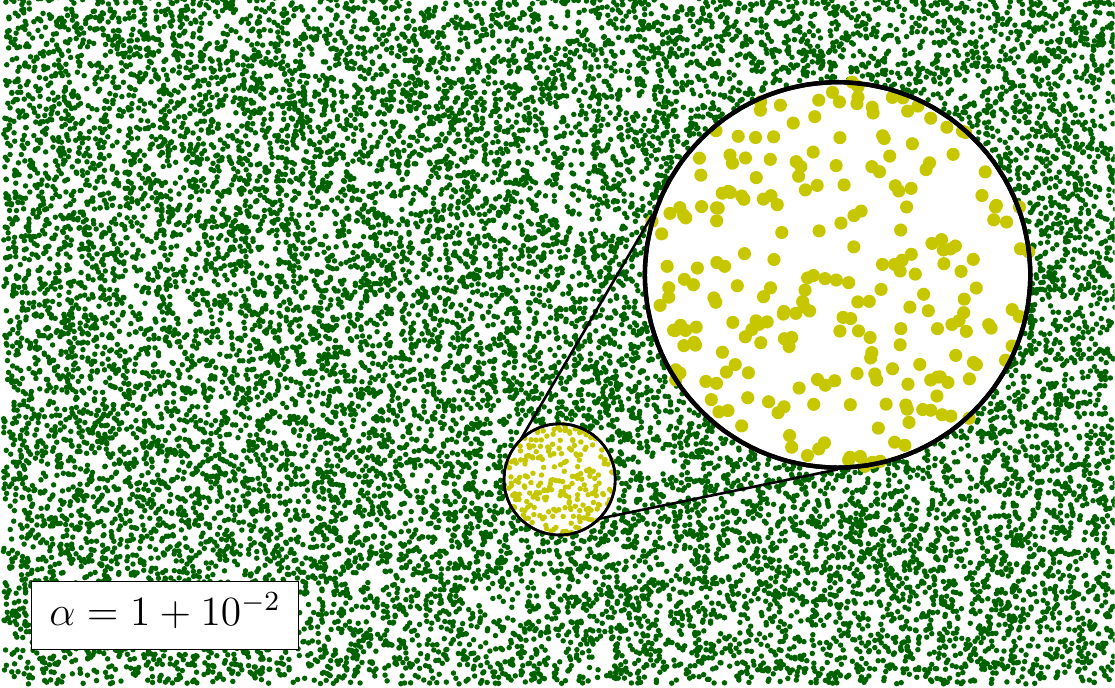}
  \hfill
  \includegraphics[width=0.49\textwidth]{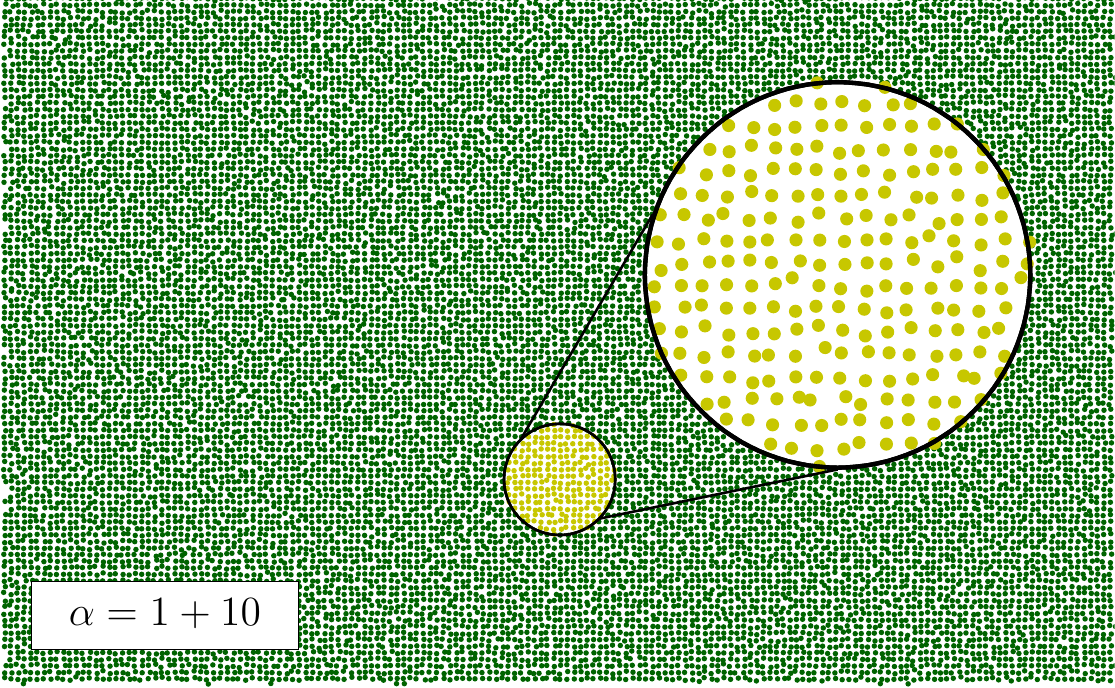}
  \caption{Hyperuniform and rigid stable matchings $\Psi^\tau$: the samples
  show matchings of a lattice on a 2D torus to two Poisson point processes with
  different intensities. Note the different degree of local order (highlighted by
  the magnifying glasses).}
  \label{fig:samples}
\end{figure}

The {simulation} procedure of the mutual nearest neighbor
matching closely follows the description of the algorithm in
Section~\ref{intro}. Given two
{point patterns}, the algorithm iteratively matches
mutual nearest neighbors and removes them from the list of unmatched
points. {When in one of the point patterns no unmatched points remain,
the algorithm terminates.}
The only difference to the infinite model discussed above is that the
finite{,} simulated samples are matched on a torus (that is, with periodic
boundary conditions).
{Our implementation uses periodic copies of points and an efficient
nearest neighbor search to match big data sets.}

For this study, we implemented the mutual nearest neighbor algorithm in
\verb+R+~\cite{R} using the \verb+RANN+ library~\cite{RANN}.
Our \verb+R+-package for matching point patterns on the
torus in arbitrary dimension is freely available {via the GitHub
repository}~\cite{Klatt2018}. 
Moreover, all data of this simulation study is {published in
the Zenodo repository~\cite{K2020}.}
The random numbers for the simulation of random point patterns are generated
by the MT19937 generator known as ``Mersenne Twister''.

Samples of a stationarized lattice $\Z^d$ in a finite
observation window $[0,L)^d$ {(considered as a flat torus)} are matched to realizations of Poisson and
determinantal point processes in 1D, 2D, and 3D with 
{varying intensities.
Figure~\ref{fig:samples} shows
2D samples of the thinned process $\Psi^{\tau}$ at intensities
either close to or far away from unity.}

{A sample of a Poisson point process with intensity $\alpha$ is simulated
by randomly placing independent and uniformly distributed points inside
an observation window, where the number of points is sampled from a Poisson
distribution with mean value $\alpha L^d$.
We simulate samples of determinantal point processes using the}
\verb+R+-function \verb+dppPowerExp+ of the \verb+spatstat+
{library~\cite{spatstat}.
It simulates a power exponential spectral model defined
in~\cite{LMR2015}.
The model parameter of our simulations are a shape parameter of 10
and a scale parameter of $(1-10^{-4})$ times the maximally allowed value.
The corresponding determinantal point process is not hyperuniform.}

In 3D, both processes were simulated at {intensities}
$\alpha = 1+10^{-2}$, $1+1$, and $1+10$,
where for the determinantal point process $L=10$, $8$, and $5$,
respectively (which corresponds to about 1000 points per sample),
and for the Poisson point process $L=22$ at each intensity
(which corresponds to about 10,000 to 100,000 points per sample).
For each process and intensity, we simulated 100~samples.

\begin{figure}[t]
  \centering
  \includegraphics[width=\textwidth]{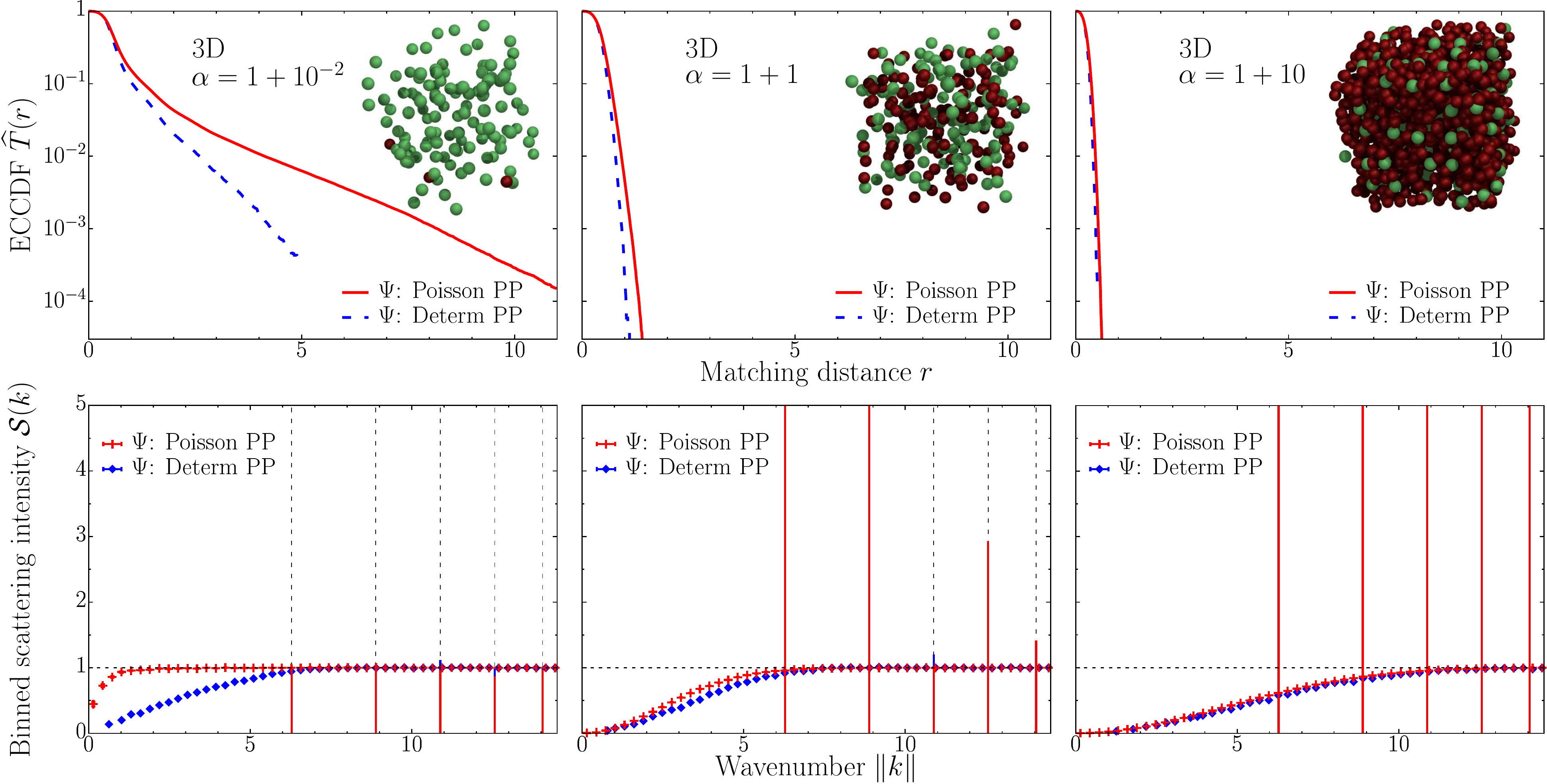}
  \caption{Stable matchings of the lattice $\Z^3$ to Poisson and
    determinantal point processes {on the torus} at different
    intensities $\alpha$ (left, center, right):
    empirical complementary cumulative distribution function
    (ECCDF) $\widehat{T}$
    of the distance $\|\tau(\mathbb{Z}^3,\Psi,q)\|$
    of a lattice point $q$ to its matching partner (top)
    and scattering intensity $\mathcal{S}(k)$ for the thinned
    process $\Psi^{\tau}$ (bottom).
    In the depicted samples of the determinantal point process,
    green (light) balls represent matched points
    and red (dark) balls represent unmatched points.}
  \label{fig:3D}
\end{figure}

Figure~\ref{fig:3D} shows (in the top) for both the Poisson and the
determinantal point process the empirical complementary cumulative
distribution function (ECCDF) $\widehat{T}$ of the distance
$\|\tau(\mathbb{Z}^3,\Psi,q)\|$ of a lattice point $q$ to its matching
partner 
{(based on all lattice points and their matching partners in all
  samples).
It is} an empirical estimate of the tail distribution
$\BP(\|\tau(\mathbb{Z}^3,\Psi,q)\|>r)$ for $r>0$.
The exponential tails in the semi-logarithmic plots are in agreement with
Theorem~\ref{t4.1} and Remark~\ref{rem:optimality},
where the constant $c_1$ depends on the process $\Psi$ and in particular on
its intensity $\alpha$.
As expected, $c_1$ decreases with decreasing $\alpha$.

Next, we analyze some two-point statistics of the process of matched points
$\Psi^{\tau}$.
The long-range behavior of the process can be well studied using the
Fourier transform of the truncated pair-correlation function.
The so-called \textit{structure factor} is defined as {the
function}
\begin{align}
  S(k):= 1 + \int_{[0,L)^d}(\rho_2(x,0)-1)e^{-i\langle k,x\rangle}\,dx, \quad k\in(2\pi/L)\N_0^d,
\end{align}
for a point process with unit intensity, where $\langle k,x\rangle$ is the scalar product of $k$ and $x$.
The norm of the \textit{wave vector} $k$ is called the
\textit{wavenumber} $\|k\|$.
{If the point process is stationary and if $\rho_2(x,0)-1$ is
(absolutetly) integrable, then} the condition \eqref{ehyperun} for
hyperuniformity is equivalent to~\cite{TorquatoStillinger2003,
KimTorquato2017}
\begin{align}
  \lim_{k\to0} S(k) = 0.
\end{align}
{In that case,} hyperuniformity can be detected
{in simulations by extrapolating $S(k)$.
However, in the case of a stationarized lattice,  
the pair-correlation function of $\Psi^{\tau}$ does not
necessarily exist.
So, instead of the structure factor, we compute a different
two-point statistics.
The empirical \textit{scattering intensity} function}
\begin{align}
  \mathcal{S}(k) := \frac{1}{N}\left|\sum_{j=1}^Ne^{-i\langle k,X_j\rangle}\right|^2
\end{align}
{is defined for  a sample of $N$ points at positions $X_j$.
The} name indicates that this quantity is directly observable in
$X$-ray or light-scattering experiments {(under suitable
assumptions)~\cite{HansenMcDonald2013}.
In statistical physics, the scattering
intensity function is commonly used to estimate the structure factor
since $\BE\mathcal{S}(k)$ (suitably compensated) converges in the limit
$L\to\infty$} to $S(k)$~\cite{HansenMcDonald2013, Torquato2018}.

Figure~\ref{fig:3D} shows (in the bottom) {the scattering intensities
$\mathcal{S}(k)$} for {the matched points in samples of}
the Poisson or the determinantal point processes in 3D at
different intensities $\alpha$.
The average values and error bars of $\mathcal{S}(k)$ correspond not
only to an average over the 100 samples per curve but also to a binning
of all allowed wave numbers, which only excludes those wave vectors that
correspond to reciprocal lattice vectors of the cubic lattice.
For the latter, the corresponding wavenumbers are indicated by dashed
vertical lines and the average values of $\mathcal{S}(k)$ are depicted by
solid vertical lines.
In the infinite system limit,
we expect these values to diverge (indicating
the underlying long-range order of the lattice).
With increasing intensity these so-called Bragg peaks are clearly
visible for both the matched Poisson and determinantal point process.
For $\alpha\to 0$, the
scattering intensities become similar to those of the unmatched point
processes $\Psi$.
In the case of a matched Poisson process, a suppression of density
fluctuations can only be seen at large length scales, that is,
small wavenumbers $\|k\|$.

For such low intensities, huge system sizes are necessary to
numerically confirm the hyperuniformity.
Therefore, we have simulated large samples of the Poisson point
process in 1D with $L=10^6$ at $\alpha=1.01$. 
We have also simulated samples in 2D with $L=300$ at $\alpha=11$.
This corresponds to about $10^6$ points per sample,
where we have simulated ten samples at each intensity.

\begin{figure}[t]
  \centering
  \includegraphics[width=0.49\textwidth]{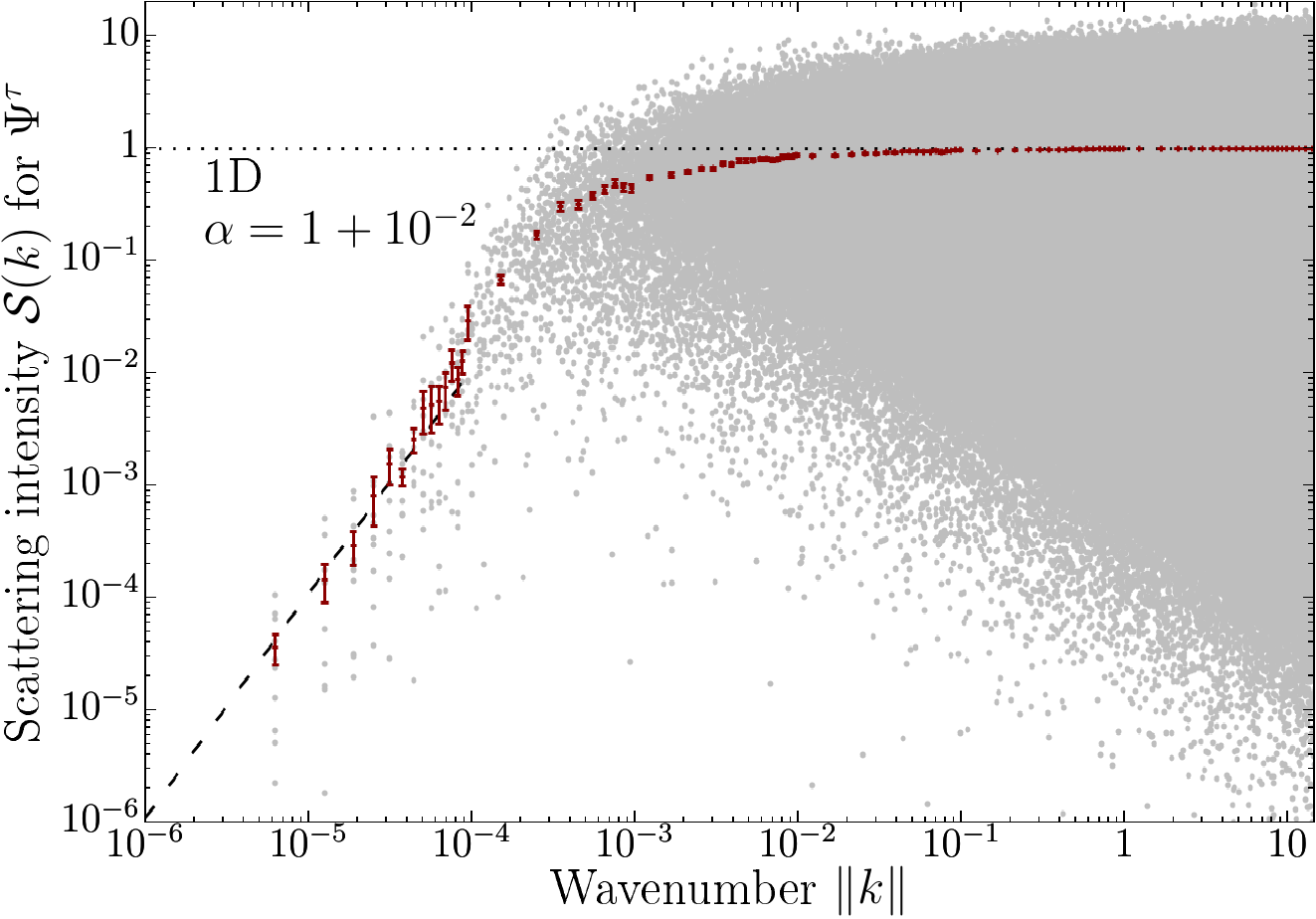}
  \hfill
  \includegraphics[width=0.49\textwidth]{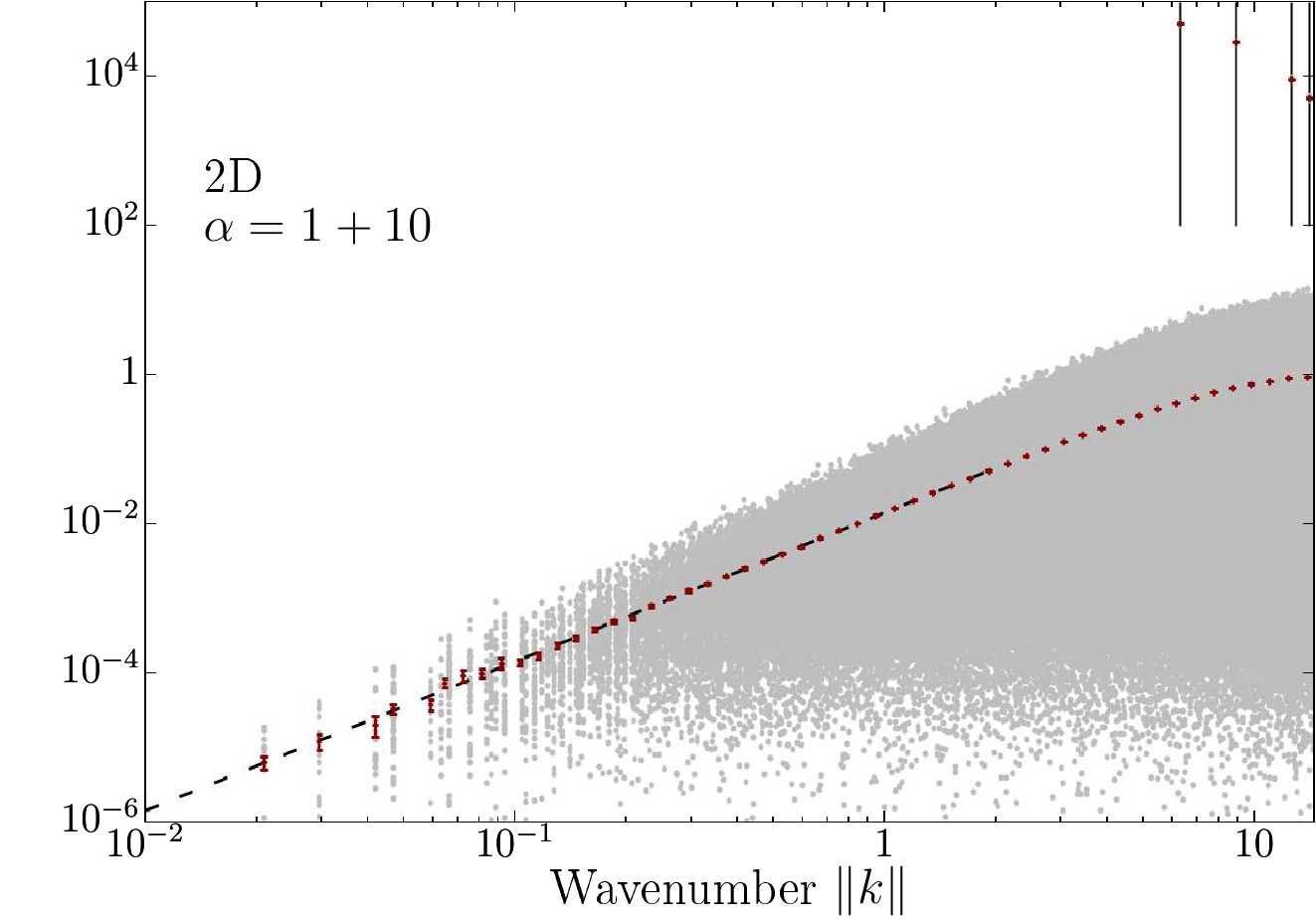}
  \caption{High precision analysis of the
  scattering intensity $\mathcal{S}(k)$ for stable
  matchings in 1D (left) and 2D (right) on the torus between a stationarized lattice
  and a Poisson point process with intensity close (left) or far away
  (right) from unity:
  The (gray) dots in the log-log plots show
  the single values of $\mathcal{S}(k)$ for allowed values of $\|k\|$
  (for all ten samples), which range over at least seven orders of
  magnitude.
  Moreover, the (red) marks with error bars depict averages over bins of
  wavenumbers.
  The dashed lines indicate a power law behavior $\|k\|^2$
  with proportionality constants $\mathcal{O}(10^6)$ (left) and
  $\mathcal{O}(10^{-2})$ (right).
  The vertical lines (right) indicate the positions of reciprocal
lattice vectors for the square lattice.}
  \label{fig:scattering}
\end{figure}

With mean values of the scattering intensity ranging over up to ten
orders of magnitude, Fig.~\ref{fig:scattering} strongly indicates
hyperuniformity of $\Psi^{\tau}$ for the matching of a Poisson point
process $\Psi$ to a stationarized lattice on the torus.
For $\alpha=1+10^{-2}$, the process $\Psi^{\tau}$ is at most wavenumbers
indistinguishable from a Poisson process.
The hyperuniformity can only be detected due to the huge system size of
a million points, where wavenumbers $k<10^{-3}$ are accessible.
For $\alpha=1+10^{2}$, the system is obviously hyperuniform.
The scattering consists of a diffuse part and peaks at the positions of the
reciprocal lattice vectors of the square lattice.
The scattering intensity is consistent with $\mathcal{S}(k)\sim \|k\|^2$
for $\|k\|\to 0$, which corresponds to a class~I hyperuniform system~\cite{Torquato2018}
in agreement with Remark~\ref{asymptotics}.

Finally, we directly estimate the \textit{number variance}
$\BV\Psi^\tau(B_R)$ of the number of matched points in a spherical
observation window $B_R$ with radius $R$.
Figure~\ref{fig:numvar} shows the number variance for a matching between
a stationarized lattice and a Poisson point process with an intensity
$\alpha=1+10^{-2}$ on the flat torus.
To detect hyperuniformity, we again have to analyze large system sizes.
So, we simulate ten samples in 2D each containing a million lattice
points (i.e., the linear system size is $L = 1000$).
We then match in each sample the lattice with the Poisson points.
To estimate the number variance, a million spherical observation windows
are randomly placed into each sample (considered as a flat torus).
For each throw, we compute the numbers of matched Poisson points within
various distances $R$ from the window center.
For a each radius $R$, the sample variance of this number is our
estimator of the number variance $\BV\Psi^\tau(B_R)$.
Its statistical error, estimated by the standard deviation of the
sample variance, is smaller than the point size in Fig.~\ref{fig:numvar}.

\begin{figure}[t]
  \centering
  $\vcenter{\hbox{\includegraphics[width=0.22\textwidth]{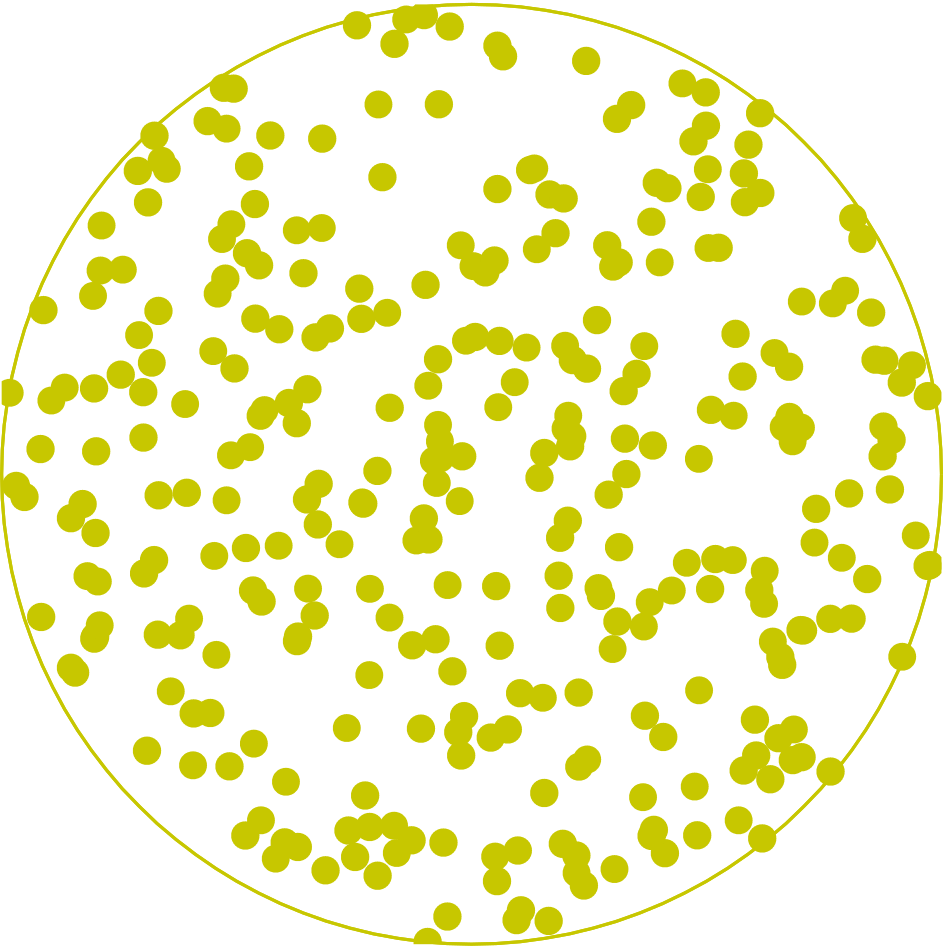}}}$
  \hfill
  $\vcenter{\hbox{\includegraphics[width=0.54\textwidth]{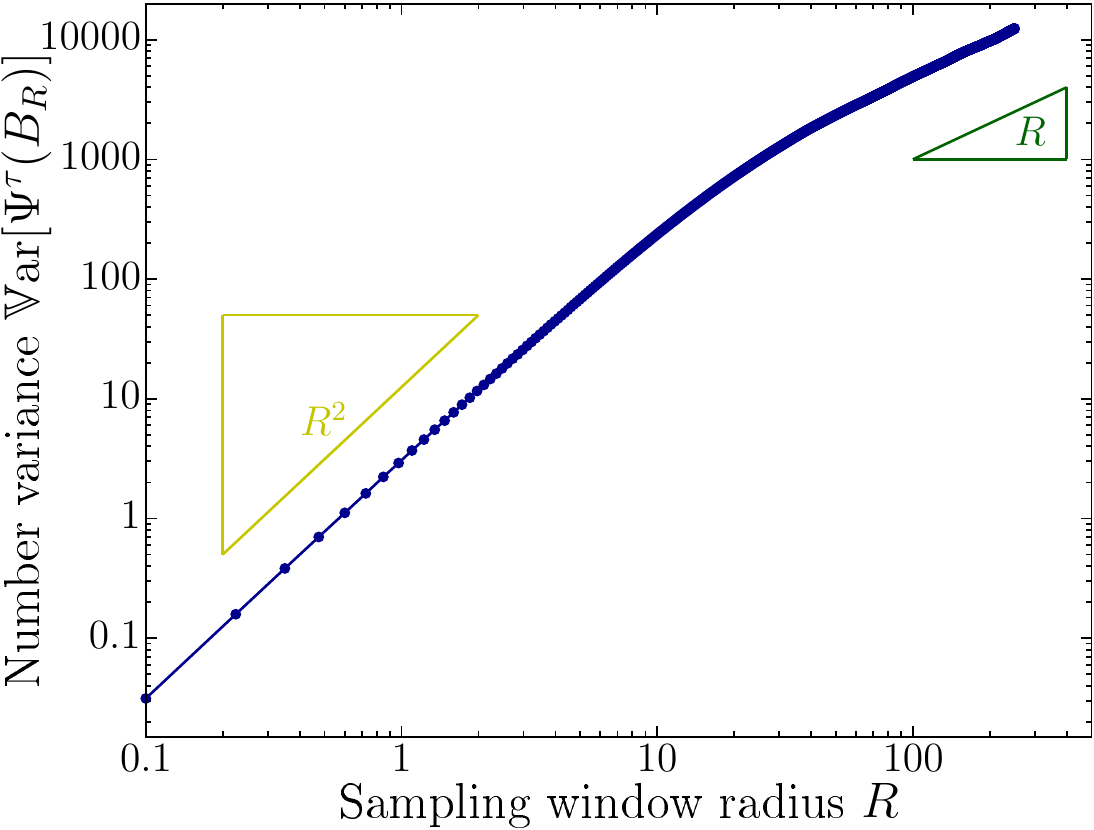}}}$
  \hfill
  $\vcenter{\hbox{\includegraphics[width=0.22\textwidth]{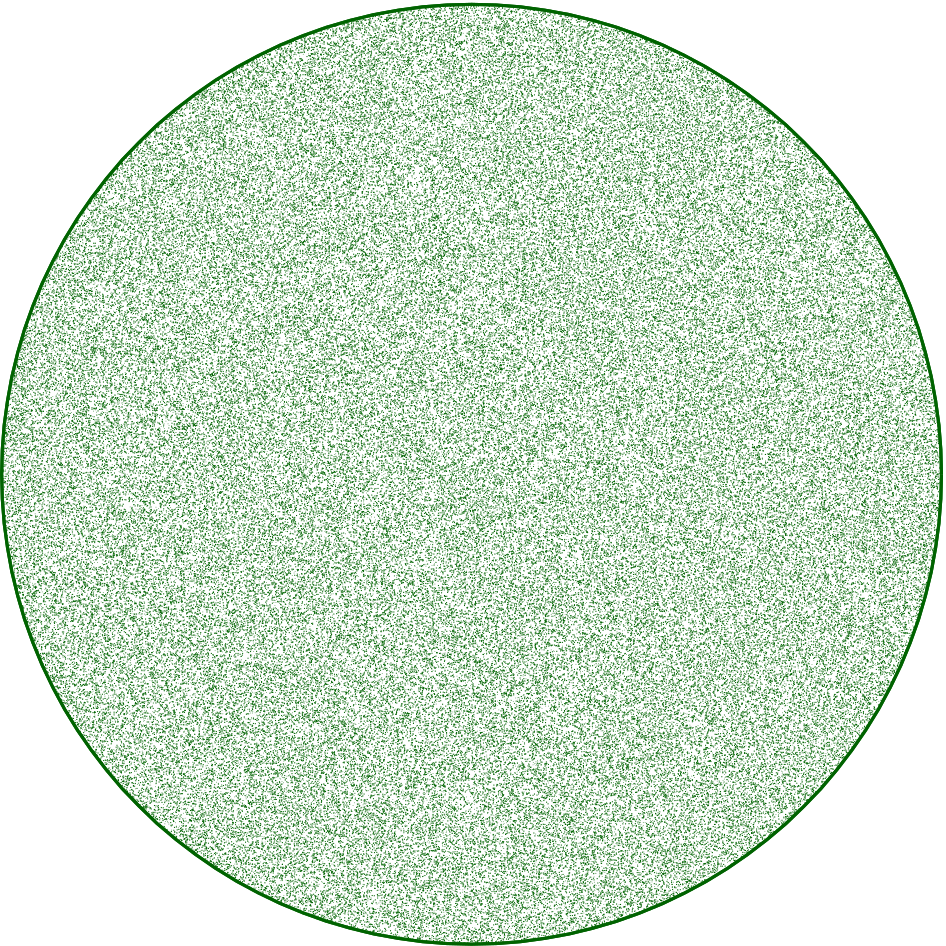}}}$
  \caption{Number variance of matched points in 2D stable matchings between a
    stationarized lattice and a Poisson point process with
    intensity $\alpha=1+10^{-2}$ on the flat torus.
    The number variance is estimated for spherical observation windows
    $B_R$ with varying radii $R$.
    Samples of observation windows are shown on the left and
    right hand side.
    They correspond to small $(R = 10)$ or large $(R = 250)$ radii.
    To highlight different scaling regimes, we use a double
    logarithmic plot.
    Up to about $R = 10$, the scaling of the number variance is
    approximately proportional to $R^2$.
    So, in this regime, the system appears to be non-hyperuniform.
    The hyperuniform, linear scaling regime can only be observed
    for radii $R > 100$.}
  \label{fig:numvar}
\end{figure}

\section{Conclusions and Outlook}

We have studied partial stable
matchings between the integer lattice and a stationary point process
of higher intensity. In one dimension, we have considered partial
one-sided stable matchings between more general stationary point
processes. Although our notion of stability is with respect to the
Euclidean distance, our techniques should work for stability
with respect to more general distances and also for more general
lattices.
Though matchings between point processes have been
studied before, the main focus has been on existence of a complete
matching between two point processes of equal intensity
and on deriving (under suitable assumptions) tail bounds on the matching distance.
Our partial matchings produce a new class of point processes that
exhibit interesting phenomena such as hyperuniformity and number
rigidity. These concepts originate in statistical physics and material
science whereas the notion of stable matching arose in combinatorial
optimization. 
Moreover, the matching can be interpreted as a spatial queueing system.
Even in 1D, it might be interesting to further study the matching from
the point of view of queueing.

Our new class of point processes provides insights into the interplay of global
and local structure in hyperuniform point processes. For any finite observation
window, we can define a point process that is within this window virtually
indistinguishable from a Poisson (or determinantal) point process.
Nevertheless, the model is rigorously hyperuniform,
where the suppression of density fluctuations occurs at arbitrarily large
length scales.
Due to the available efficient algorithms for the nearest neighbor
search, our stable matching offers a straightforward
and efficient simulation of huge correlated, hyperuniform point patterns
with a tunable degree of local order,
see Figs.~\ref{fig:samples} and \ref{fig:scattering}
and the supplementary video.

Our results can be further pursued and possibly extended in several
directions.  First of all, we believe that hyperuniformity and
rigidity remain true if $\Phi$ is a stationarized lattice
and $W$ the unit ball
(which is confirmed by our simulations). In fact it
is tempting to conjecture that $\Psi^\tau$ is hyperuniform, whenever
$\Phi$ is a hyperuniform point process. Even beyond that
$\Psi^\tau$ might inherit any asymptotic number variance from
$\Phi$.
An intriguing example is if $\Phi$ is anti-hyperuniform or hyperfluctuating, that is, 
if $\lim_{r\to\infty}{\BV\Psi^\tau(rW)}/{\lambda_d(rW)}=\infty$ like for a Poisson hyperplane
intersection process~\cite{HSS2006}.
As evidence, we have shown
the above in the special case of one dimension for one-sided matchings in
Section~\ref{secone-sided}. It is of interest to study whether these
point processes exhibit further rigidity (see~\cite{Ghosh2015} for
more about notions of rigidity).
Simply put our simulations suggest that the matched point process
``inherits'' local properties from $\Psi$ and global properties from
$\Phi$ if the intensity of $\Psi$ is larger than that of $\Phi$.

Further, our proofs of hyperuniformity and number rigidity relied
crucially upon exponential tail bounds for the typical matching
distance. In light of this, one can ask what do exponential tail
bounds for the typical matching distance between two i.i.d.~copies of the
same point process imply about the point process.

\section{Appendix: Some point process results}
\label{s:app}

We shall briefly introduce here some basic point process notions and
state some general results of use for us. For $\varphi \in \bN$ and $k \in
\N$, we define the {\em $k$th factorial measure} of $\varphi$ as
\[ \varphi^{(k)} := \{(x_1,\ldots,x_k) : x_i \in \varphi, x_i \neq x_j, \forall i \neq j \}. \]
As with $\varphi$, we often view $\varphi^{(k)}$ as a measure on $(\R^d)^k$ as follows :
\[ \varphi^{(k)}(.) := \sideset{}{^{\ne}}\sum_{x_1,\ldots,x_k \in \varphi} \delta_{(x_1,\ldots,x_k)}(.),\]
where $\sum^{\neq}$ denotes that the summation is taken over pairwise
distinct entries and an empty sum is defined to be $0$. For a point
process $\Psi$ and $k \in \N$, we define the {\em $k$th factorial moment
measure} as $\alpha^{(k)}(.) = \BE(\Psi^{(k)}(.))$. If the measure
$\alpha^{(k)}$ has a density with respect to the Lebesgue measure, i.e.,
for all $x_1,\ldots,x_k \in \R^d$, $\alpha^{(k)}(\md x_1 \ldots \md x_k)
= \rho_k(x_1,\ldots,x_k)\md x_1 \ldots \md x_k$, then $\rho_k$ is said
to be the {\em $k$-point correlation function} of $\Psi$. 

Given a point process $\Psi$, we can define the following mixing coefficient
\begin{align*}
\alpha_{p,q}(s) &:= \sup \left\{ |\COV{f(\Psi \cap A),g(\Psi \cap B)}| : f, g \in \{0,1\},  \right.\\
          &\hspace*{4.5cm} \left. \lambda_d(A) \leq p, \lambda_d(B) \leq q, d(A,B) \geq s \right\},
\end{align*}
where $d(A,B) := \inf_{x \in A, y \in B} \|x-y\|$.

Let $K : \R^d \times \R^d \to \mathbb{C}$ be a measurable function. We
call such a $K$ {\em kernel.} We assume that the kernel $K$ is
Hermitian, non-negative definite, locally square-integrable and the
associated integral operator is locally trace-class with eigenvalues
in $[0,1]$. A point process $\Psi$ is said to be {\em a determinantal
  point process with kernel $K$} if for all $k \geq 1$ and
$x_1,\ldots,x_k \in \R^d$, we have that
\[ \rho_k(x_1,\ldots,x_k) = \det\left( (K(x_i,x_j))_{1 \leq i,j \leq k} \right) .\]
The above conditions on the kernel $K$ guarantee the existence and
uniqueness of the determinantal point process with kernel $K$ (see
Theorem 4.5.5 in \cite{Hough09}.) Further if $\Psi$ is stationary,
$K(x,x)$ is a constant and is the intensity of the process $\Psi$. As
a simple consequence of Hadamard's inequality, we derive the following
useful inequality :
\begin{equation}
\label{e:Had}
\rho_k(x_1,\ldots,x_n) \leq K(0,0)^n.
\end{equation}

Though the $k$-point correlation functions of the stationary
Poisson point process satisfy the condition of being determinantal
with the kernel $K(x,y) = \1\{|x-y| = 0\}$, the associated integral
operator is degenerate. Therefore, the Poisson point
process is, strictly speaking, excluded
from our framework for determinantal point processes.
However, the stationary Poisson point process has all
properties that we use of determinantal point processes (such as
\eqref{e:Had}, Theorems \ref{t:dpp_conc} and \ref{t:dpp_mixing}).
Hence, we include it as an example of
determinantal point processes.

We now state an important concentration inequality for determinantal
point processes as well as a bound for the $\alpha_{p,q}$ mixing
coefficient.  The next two results can be found as Theorem 3.6 in
\cite{Pemantle14} and Corollary 4.2 in \cite{Poinas17} respectively.
\begin{theorem}
\label{t:dpp_conc}
Let $\Psi$ be a stationary determinantal point process (in particular Poisson) with kernel $K$ and intensity
$\alpha = K(0,0)$ and $f : \bN \to \R$ be a $c$-Lipschitz functional
(i.e., $|f(\varphi + \delta_x) - f(\varphi)| \leq c$ for all
$x \in \R^d, \varphi \in \bN$). Then for any $a > 0$ and bounded set $A$,
we have that
\begin{align*}
\BP (|f(\Psi \cap A) - \BE(f(\Psi \cap A))| \geq a) 
\leq 5 \exp\Big(-\frac{a^2}{4c(a + 2\alpha\lambda_d(A)c)}\Big).
\end{align*}
\end{theorem}
\begin{theorem}
\label{t:dpp_mixing}
Let $\Psi$ be a stationary determinantal point process  (in particular Poisson) with kernel $K$ and define $\omega(s) := \sup_{x,y \in \R^d, \|x-y\| \geq s} |K(x,y)|$, $s >0$. Then we have that for all $s > 0$, 
$$\alpha_{p,q}(s) \leq pq \omega(s)^2$$
\end{theorem}
It can be shown that if for all $p,q >0$, $\alpha_{p,q}(s) \to 0$ as $s
\to \infty$, then $\Psi$ is mixing. For example, see Section 3 in
\cite{Ivanoff82}.

\section*{Acknowledgments}

This work was in part supported by the German Research Foundation
(Deutsche For\-schungsgemeinschaft) through Grants No. HU1874/3-2 and
No. LA965/6-2 awarded as part of the DFG-Forschergruppe FOR 1548
``Geometry and Physics of Spatial Random Systems.'' DY's work was
supported by INSPIRE Faculty Award from DST and CPDA from the Indian
Statistical Institute.  This project was initiated during DY's visit
to the Institute of Stochastics, Karlsruhe Institute of Technology and
DY wishes to thank the institute for hosting him. DY would like to
thank Rapha\"el Lachi\`eze-Rey for his comments on an earlier draft and
Subhroshekhar Ghosh and Manjunath Krishnapur for discussions on higher-dimensional examples of rigid point processes. The authors are thankful to the two anonymous referees for their comments leading to an improved presentation.

\end{document}